\documentclass[11pt,reqno]{amsart}
\usepackage{amsmath} 
\usepackage{amssymb}
\usepackage{dsfont}
\usepackage[dvips,draft,final]{graphics}
\usepackage{amssymb,amsmath}
\usepackage[T1]{fontenc}
\usepackage{fancyhdr}

\def\squarebox#1{\hbox to #1{\hfill\vbox to #1{\vfill}}}

\theoremstyle{plain}

\newtheorem{Thm}{Theorem}

\newtheorem{lem}{Lemma}

\newcommand{\im}{\textrm{Im}}
\newcommand{\Div}{\textrm{div}}
\newcommand{\R}{\mathbb{R}}
 
\newcommand{\B}{\dot{\mathcal{H}}_1({\R}^n)}

\newcommand{\CI}{{\mathcal C}^{\infty}_{0}({\R}^{n}) }
\newcommand{\C}{{\mathbb C}}

\newcommand{\half}{\frac{1}{2}}

\newcommand{\BI}{H^{1,A_1}_{b,s}({\R}^{1+n})}
\newcommand{\CII}{{\mathcal C}^{\infty}_{0}(\vert x\vert\leq b) }
\def\epsilon{\varepsilon}
\def\phi {\varphi}

\newtheorem{rem}{Remark}
\newtheorem{prop}{Proposition}
\newtheorem{defi}{Definition}

\numberwithin{equation}{section}
\renewcommand{\d}{\textrm{d}}
\renewcommand{\leq}{\leqslant}
\renewcommand{\geq}{\geqslant}
\providecommand{\abs}[1]{\left\lvert#1\right\rvert}
\providecommand{\norm}[1]{\left\lVert#1\right\rVert}

\begin{document}

\title[Local energy decay ]
{ Local energy decay in even dimensions for the wave equation  with a time-periodic non-trapping metric  and applications to Strichartz
estimates}

\author[Y. Kian]{Yavar Kian}

\address {Universit\'e Bordeaux I, Institut de Math\'ematiques de Bordeaux,  351, Cours de la Lib\'eration, 33405  Talence, France}

\email{Yavar.Kian@math.u-bordeaux1.fr}
\maketitle
\begin{abstract}
We obtain local energy decay as well as global Strichartz estimates for the solutions $u$ of the wave equation
 $\partial_t^2 u-div_x(a(t,x)\nabla_xu)=0,\  t\in{\R},\ x\in{\R}^n,$ with time-periodic non-trapping metric $a(t,x)$ equal to $1$ outside a compact set with respect to $x$. We suppose that the cut-off resolvent
  $R_\chi(\theta)=\chi(\mathcal U(T, 0)-e^{-i\theta})^{-1}\chi$, where $\mathcal U(T, 0)$ is the monodromy operator and $T$ the period of $a(t,x)$, admits an holomorphic continuation to   
$\{\theta\in\mathbb{C}\  :\   \textrm{Im}(\theta) \geq  0\}$, for $n \geq 3$ , odd, and to $\{ \theta\in\mathbb C\ :\ \textrm{Im}(\theta)\geq0,\ \theta\neq 2k\pi-i\mu,\ k\in\mathbb{Z},\ \mu\geq0\}$ for $n \geq4$, even, and for $n \geq4$ even $R_\chi(\theta)$ is bounded in a neighborhood of $\theta=0$. 
\end{abstract}

\section{Introduction}
\renewcommand{\theequation}{\arabic{section}.\arabic{equation}}
\setcounter{equation}{0}

Consider the Cauchy problem
\begin{equation} \label{eq=lepbA}  \left\{\begin{array}{c}
u_{tt}-\Div_{x}(a(t,x)\nabla_{x}u)=0,\ \ (t,x)\in{\R}\times{\R}^{n},\\
(u,u_{t})(s,x)=(f_{1}(x),f_{2}(x))=f(x),\ \ x\in{\R}^n,\end{array}\right.\end{equation}
where the perturbation $a(t,x)\in \mathcal C^\infty({\R}^{n+1})$ is a scalar function which satisfies the conditions:
\begin{equation}\label{eq=perturbationA}\begin{array}{l}(i) \ C_0\geq a(t,x)\geq c_0>0,\   (t,x)\in{\R}^{n+1},\\
(ii)\ \textrm{ there exists }\rho>0\textrm{ such that }a(t,x)=1\textrm{ for }\vert x\vert\geq\rho,\\
(iii)\textrm{ there exists }T>0\textrm{ such that }a(t+T,x)=a(t,x),\   (t,x)\in{\R}^{n+1}.\\
\end{array}\end{equation}
Throughout this paper we assume   $n\geq3$. Let  $\dot{H}^\gamma({\R}^n)=\Lambda^{-\gamma} (L^2({\R}^n))$ be the homogeneous Sobolev spaces, where $\Lambda=\sqrt{-\Delta_x}$ is determined by the Laplacian in ${\R}^n$. The solution of (1.1) is given by the propagator
\[\mathcal{U}(t,s):\dot{\mathcal{H}}_\gamma({\R}^n)\ni(f_1,f_2)=f\mapsto \mathcal{U}(t,s)f=(u,u_t)(t,x)\in\dot{\mathcal{H}}_\gamma({\R}^n)\]
where $\dot{\mathcal{H}}_\gamma({\R}^n)=\dot{H}^\gamma({\R}^n)\times \dot{H}^{\gamma-1}({\R}^n)$.
Our goal in  this paper is to establish  that for cut-off functions $\psi_1,\psi_2\in\mathcal C^\infty_0(\vert x\vert\leq\rho+1)$ ,  we have local energy decay having the form
\begin{equation} \label{eq=locA}\Vert\psi_1 \mathcal U(t,s)\psi_2\Vert_{\mathcal L({\B})}\leq C_{\psi_1,\psi_2}p(t-s),\quad t\geq s,\end{equation}
with $p(t)\in L^1({\R}^+)$.  For this purpose,  we assume that the perturbation associated to $a(t,x)$ is non-trapping. More precisely, consider the null bicharacteristics $(t(\sigma),x(\sigma),\tau(\sigma),\xi(\sigma))$ of the principal symbol $\tau^2-a(t,x)\vert\xi\vert^2$ of $\partial^2_t-\Div_x(a\nabla_x)$ satisfying
\[t(0)=t_0,\vert x(0)\vert\leq R_1,\: \xi(0) = \xi_0,\quad \tau^2(\sigma)=a(t(\sigma),x(\sigma))\vert\xi(\sigma)\vert^2.\]
It is known that for $\xi_0\neq 0$, the null bicharacteristics can be parametrized with respect to $t$ and they can be defined for $t\in{\R}$ (see \cite{K1}). We denote by $(t,x(t),\tau(t),\xi(t))$ the bicharacteristic $(t(\sigma),x(\sigma),\tau(\sigma),\xi(\sigma))$ parametrized  with respect to $t$.
We introduce the following condition
\begin{enumerate}
\item[$\rm(H1)$]
 We say that the metric $a(t,x)$ is non-trapping if for all $R>R_1$ there exists $T(R,R_1)>0$ such that $\vert x(t)\vert>R$ for $\vert t-t_0\vert\geq T(R,R_1)$. 
\end{enumerate}
The non-trapping condition (H1) is necessary for  \eqref{eq=locA}  since for some trapping perturbations we may have solutions with exponentially increasing local energy (see \cite{CR}). On the other hand, even for non-trapping periodic perturbations some parametric resonances could lead to solutions with exponentially growing local energy (see \cite{CPR} for the case of time-dependent potentials). To exclude the existence of such solutions we must impose a second hypothesis.
Let\\
 $\psi_1,\psi_2\in{\CI}$. We define the cut-off resolvent associated to problem (\ref{eq=lepbA}) by \[R_{\psi_1,\psi_2}(\theta)=\psi_1(\mathcal{U}(T,0)-e^{-i\theta})^{-1}\psi_2.\] Consider the following assumption.

\begin{enumerate}
\item[$\rm(H2)$] Let $\psi_1,\psi_2\in{\CI}$ be such that $\psi_i = 1$ for $|x| \leq \rho + 1 + 3T, i = 1, 2$. Then the operator $R_{\psi_1,\psi_2}(\theta)$ admits a holomorphic continuation from $\{\theta\in\mathbb{C}\  :\   \textrm{Im}(\theta) \geq A > 0\}$ to
$\{\theta\in\mathbb{C}\  :\   \textrm{Im}(\theta) \geq  0\}$, for $n \geq 3$, odd, and to $\{\theta\in\mathbb{C}\  :\   \textrm{Im}(\theta) >  0\}$ for $n \geq4$, even.
Moreover, for $n$ even, $R_{\psi_1,\psi_2}(\theta)$ admits a continuous continuation from $\{\theta\in\mathbb{C}\  :\   \textrm{Im}(\theta) > 0\}$ to $\{\theta\in\mathbb{C}\  :\   \textrm{Im}(\theta) \geq  0,\theta\neq 2k\pi, k\in\mathbb{Z}\}$ and we have
\[\limsup_{\substack{\lambda\to0 \\ \textrm{Im}(\lambda)>0}}\Vert R_{\psi_1,\psi_2}(\lambda)\Vert<\infty.\]
\end{enumerate}
We like to mention that in the study of the time-periodic perturbations of the Schr\"odinger operator
(see [3] ) the resolvent of the monodromy operator $(\mathcal U(T) - z)^{-1}$ plays a central role. Moreover, the absence of eigenvalues $z \in\mathbb C, |z| = 1$ of $\mathcal U(T)$, and the behavior of the resolvent for $z$ near $1$, are closely related to the decay of local energy as $t \rightarrow\infty$. So our results may be considered as a natural extension of those for Schr\"odinger operator. On the other hand, for the
wave equation we may have poles $\theta \in\mathbb C$, $  \textrm{Im}(\theta) > 0$ of the resolvent $R_{\psi_1,\psi_2}(\theta)$, while for the Schr\"odinger operator
with time-periodic potentials a such phenomenon is excluded.

Many authors  have investigated the  local  energy decay as well as $L^2$ integrability for the local energy of wave equations. The results of microlocal analysis concerning the  propagation of singularities make possible to improve  many results of local energy decay.  Tamura also established several results about the local energy decay  (see \cite{tam}, \cite{T7}, \cite{T6}, \cite{T5}). Assume  $n=3$ and let $V(t,x)$ satisfy the conditions
 
 \[\begin{array}{l}(i) V(t,x) \textrm{ is non-negative and $\mathcal C^1$ with uniformly bounded derivative}, \\
(ii)\ \textrm{ there exists  }\rho>0\textrm{ such that }V(t,x)=0\textrm{ for }\vert x\vert\geq\rho,\\
(iii)\partial_tV(t,x)=\underset{t\to+\infty}O(t^{-\alpha})\textrm{ for a } 0<\alpha\leq1,\textrm{ uniformly in $x$}. \end{array}\]
Then, in \cite{T5} Tamura shows that the local energy of the solution of  the wave equation  \[\partial_t^2-\Delta_xu+V(t,x)u=0,\]  decreases exponentially. In \cite{tam}, Tamura  also obtains a  decay of the local energy associated to  \eqref{eq=lepbA}, when the metric $a(t,x)$ is independent of $t$ and  admits a discontinuity, by applying arguments similar to the those used in \cite{MRS}. 

By using the compactness of the local evolution operator, deduced from a propagation of singularities, and the RAGE theorem of Georgiev and Petkov (see \cite{GP}), Bachelot and Petkov show in \cite{BP1} that in the case of odd dimensions, the decay of the local energy associated to the wave equation with time periodic potential  is exponential for initial data  with compact support  included in a subspace of finite codimension.

In \cite{V3} and \cite{V1}, Vainberg proposed a general analysis of problems with non-trapping  perturbations investigating the asymptotic behavior of the cut-off resolvent. Using the same approach,  Vodev has established in \cite{V5} and \cite{V4} that, for non-trapping perturbation $a(x)$ and for $n\geq4$ even, we have \eqref{eq=locA} with  $p(t)=t^{1-n}$.

Notice that all these results are based on the analysis of the Fourier transform with respect to $t$ of the solutions. Since the principal symbol of $\partial_t^2-\Div_x(a(t,x)\nabla_x\cdot)$ is time dependent, we can not use this argument for \eqref{eq=lepbA}.

To obtain (\ref{eq=locA}), we use the assumption (H2). For $n\geq3$, odd, we obtain an exponential decay of energy with $p(t)=e^{-\delta t}$, $\delta>0$. For $n\geq4$ even, it is more difficult to prove (\ref{eq=locA}) and we  use the results  of Vainberg for non-trapping, time-periodic problems.
In \cite{V2} Vainberg proposed a general approach to problems with time-periodic perturbations including potentials, moving obstacles and high order operators, provided that the perturbations are non-trapping. 
The analysis in \cite{V2} is based on the meromorphic continuation of an operator $R(\theta)$ introduced in Section 2. In order to obtains (\ref{eq=locA}) for $n\geq4$ even, we will establish the link between $R_{\psi_1,\psi_2}(\theta)$ and the operator $R(\theta)$. Our main result is the following
\begin{Thm}\label{t3}
Assume $\rm(H1)$ and $\rm(H2)$  fulfilled and  $n\geq4$  even. Let $\chi_1,\psi_1\in\mathcal C_0^\infty(\vert x\vert\leq\rho+1)$. Then, for all $s\leq t$, we have
\begin{equation}\label{eq=thm2A} \Vert \chi_1\mathcal U(t,s)\psi_1\Vert_{\mathcal L({\B})}\leq Cp(t-s),\end{equation}
 where $C>0$ is independent of $t$, $s$ and   $p(t)$ is  defined by
\[p(t)= \frac{1}{(t+1)\ln^2(t+e)}.\]
\end{Thm}

Let $U_0(t)$ be the unitary group on ${\B}$ related to the Cauchy problem (1.1) for the free wave equation  ($a=1$ and $\tau=0)$. For $b\geq\rho$ denote by $P^b_+$ (resp $P^b_-$)  the orthogonal projection on the orthogonal complements of the Lax-Phillips spaces 
\[D^b_\pm=\{f\in{\B}\ :\ (U_0(t)f)_1(x)=0\  \textrm{ for } \ \vert x\vert<\pm t+b\}.\]
Set $Z^b(t,s)=P^b_+\mathcal{U}(t,s)P^b_-$. Then, for $n$ odd, the resonances of the problem (1.1) coincide with the eigenvalues of the operator $Z^b(T,0)$ and the condition (H1) guarantees that the the spectrum $\sigma(Z^b(T,0))$ of $Z^b(T,0)$ is formed by eigenvalues $z_j\in\mathbb{C}$ with finite multiplicities. Moreover, these eigenvalues are independent of the choice of $b\geq\rho$ (see for more details [18] for time-periodic potentials and moving obstacles). Consider the following assumption
\begin{enumerate}
\item[$\rm(H3)$]\[\sigma(Z^b(T,0))\cap\{z\in\mathbb{C}\ :\ \vert z\vert\geq1\}=\varnothing.\]\end{enumerate}

In \cite{K1}, we have established global Strichartz estimates for $n\geq 3$ odd, assuming that (H1) and (H3) are fulfilled.
 More precisely, (H1) and  (H3) imply that   for $n\geq3$ odd and for $2\leq p,q<+\infty$ satisfying
\begin{equation}\label{eq=stA}p>2,\quad \frac{1}{p}=n\left(\half-\frac{1}{q}\right)-1,\quad\textrm{and}\quad \frac{1}{p}< \frac{n-1}{2}\left(\half-\frac{1}{q}\right),\end{equation}
the solution $u(t)$ of (\ref{eq=lepbA}), for $s=0$ and $f\in{\B}$, satisfies the estimate
\begin{equation}\label{eq=estA}\Vert u(t)\Vert_{L_t^p({\R}^+,L_x^q({R}^n))}+ \Vert u\Vert_{L_t^\infty({\R}^+,\dot{H}^1({\R}^n))}+\Vert  u_t\Vert_{L_t^\infty({\R}^+,L^2({\R}^n))}\leq C(\rho,T,n,p,q)(\Vert f_1\Vert_{\dot{H}^1({\R}^n)}+\Vert f_2\Vert_{L^2({\R}^n)}).\end{equation}
Moreover, in \cite{K1}, it has been proved  that for $n\geq3$, for  $2\leq p,q<+\infty$ satisfying
\begin{equation}\label{eq=stB}\quad \frac{1}{p}=n\left(\half-\frac{1}{q}\right)-1\quad\textrm{and}\quad \frac{1}{p}< \frac{n-1}{2}\left(\half-\frac{1}{q}\right),\end{equation}
 for  the solution $u$ of (\ref{eq=lepbA}) we have the following local estimate
\begin{equation}\label{eq=estB}\Vert\chi u\Vert_{L^p_t([s,s+\delta],L_x^q({\R}^n))}\leq C\Vert f\Vert_{\B}\end{equation}
with $C,\delta>0$ independent of $f$ and $s$, and $\chi\in{\CI}$.
Applying \eqref{eq=thm2A}, we show that the estimates (\ref{eq=estA}) remain  true for even dimensions and we obtain the following
\begin{Thm}\label{t1}
Assume $n\geq3$  and let $a(t,x)$ be a metric such that $\rm(H1)$ and $\rm(H2)$ are fulfilled.  Let
 $2\leq p,q<+\infty$ satisfy conditions \eqref{eq=stA}.
Then for the solution $u(t)$ of \eqref{eq=lepbA} with $s=0$ and $f\in{\B}$ we have  the estimate
\begin{equation} \label{eq=estC}\Vert u(t)\Vert_{L_t^p({\R}^+,L_x^q({R}^n))}+ \Vert u\Vert_{L_t^\infty({\R}^+,\dot{H}_x^1({\R}^n))}+\Vert u_t\Vert_{L_t^\infty({\R}^+,L_x^2({\R}^n))}\leq C(p,q,\rho,T)(\Vert f_1\Vert_{\dot{H}^1({\R}^n)}+\Vert f_2\Vert_{L^2({\R}^n)}).\end{equation}
\end{Thm}
The results of Theorem \ref{t1} have been exploited in \cite{K2} to prove the existence of local weak solutions of semilinear wave equations for small initial data and long time intervals.

Notice that the estimate
\begin{equation} \label{a}\Vert\psi_1 \mathcal U(NT,0)\psi_2\Vert_{\mathcal L({\B})}\leq \frac{C_{\psi_1,\psi_2}}{(N+1)\ln^2(N+e)},\quad N\in\mathbb N,\end{equation}
implies (\ref{eq=locA}). On the other hand, if (\ref{a}) holds, the assumption (H2) for $n$ even is fulfilled. Indeed, for large $A>>1$ and $ \textrm{Im}(\theta)\geq AT$ we have
\[R_{\psi_1,\psi_2}(\theta)=-e^{i\theta}\sum_{N=0}^{\infty}\psi_1\mathcal U(NT,0)\psi_2e^{iN\theta}\]
and applying (\ref{a}), we conclude that $R_{\psi_1,\psi_2}(\theta)$ admits a holomorphic continuation from\\
 $\{\theta\in\mathbb{C}\  :\   \textrm{Im}(\theta) \geq A > 0\}$   to $\{\theta\in\mathbb{C}\  :\   \textrm{Im}(\theta) >  0\}$. Moreover, $R_{\psi_1,\psi_2}(\theta)$ is bounded for $\theta\in{\R}$. In Section 4, we give some examples of metrics $a(t,x)$ such that  (\ref{a}) is fulfilled.  
\begin{rem}
Let the metric $(a_{ij}(t,x))_{1\leq i,j\leq n}$ be such that for all $i,j=1\cdots n$ we have
\[\begin{array}{l}
\displaystyle(i)\ \textrm{ there exists }\rho>0\textrm{ such that }a_{ij}(t,x)=\delta_{ij},\textrm{ for }\vert x\vert\geq\rho,\textrm{ with $\delta_{ij}=0$ for $i\neq j$ and $\delta_{ii}=1$},\\
\displaystyle(ii)\textrm{ there exists }T>0\textrm{ such that }a_{ij}(t+T,x)=a_{ij}(t,x),\    (t,x)\in{\R}^{n+1},\\
\displaystyle(iii)a_{ij}(t,x)=a_{ji}(t,x),  (t,x)\in{\R}^{n+1},\\
\displaystyle(iv)\textrm{ there exist } C_0>c_0>0 \textrm{ such that }C_0\vert\xi\vert^2\geq\sum_{i,j=1}^na_{ij}(t,x)\xi_i\xi_j\geq c_0\vert\xi\vert^2 ,\ \  (t,x)\in{\R}^{1+n},\ \xi\in{\R}^n.\\
\end{array}\]

If we replace $a(t,x)$ by $(a_{ij}(t,x))_{1\leq i,j\leq n}$ in \eqref{eq=lepbA} we get the following problem
\begin{equation} \label{me}  \left\{\begin{array}{c}
\displaystyle u_{tt}-\sum_{i,j=1}^n\frac{\partial}{\partial x_i}\left(a_{ij}(t,x)\frac{\partial}{\partial x_j}u\right)=0,\ \ (t,x)\in{\R}^{n+1},\\
\ \\
(u,u_{t})(s,x)=(f_{1}(x),f_{2}(x))=f(x),\ \ x\in{\R}^n.\end{array}\right.\end{equation}

Repeating  the  argument  for   \eqref{eq=lepbA} we can prove that estimates \eqref{eq=thm2A} and \eqref{eq=estC} are  true for the solution $u$ of the problem \eqref{me} if for the trajectories of the symbol $\tau^2-\sum_{i,j=1}^na_{ij}(t,x)\xi_i\xi_j$ and the corresponding resolvant $R_{\phi_1,\phi_2}(\theta)$, $\rm(H1)$ and $\rm(H2)$ are fulfilled and if $n,p,q$ satisfied \eqref{eq=stA}.
\end{rem}
\vspace{0,5cm}
\ \\
 \textbf{Acknowledgements.} The author would like to thank Vesselin Petkov for his precious help during the preparation of this work, Jean-Fran\c{c}ois Bony  for  his remarks and the referee for his suggestions.

\section{ Equivalence of $\rm(H2)$ and $\rm(H3)$  for odd dimensions}

In this section, we assume that (H1) is fulfilled and our purpose is to prove that for $n$ odd the assumptions (H2) and (H3) are equivalent. We start by recalling some properties of the operators $Z^b(t,s)$ and $\mathcal U(t,s)$.
\begin{prop}\label{p1}
For all $s,t\in{\R}$, we have
  \begin{equation}\label{eq=prop1A}\mathcal{U}(t+T,\tau+T)=\mathcal{U}(t,\tau).\end{equation}
\end{prop}
Applying (\ref{eq=prop1A}), we get 
 \[\mathcal{U}((N+1)T,NT)=\mathcal{U}(T,0),\quad N\in\mathbb{N}.\]
and it follows $(\mathcal{U}(T,0))^N=\mathcal{U}(NT,0),\quad N\in\mathbb{N}.$
From now on, we set $\mathcal U(T)=\mathcal U(T,0).$

\begin{prop}\label{p2}
Assume $n\geq2$ and let $a(t,x)$ satisfy \eqref{eq=perturbationA}. Then, we get
\begin{equation}\label{eq=prop2A}\Vert \mathcal{U}(t,s)\Vert_{\mathcal{L}({\B})}\leq Ce^{A \vert t-s\vert},\end{equation}
where \[A=\left\Vert\frac{a_t}{a}\right\Vert_{L^\infty({\R}^{1+n})}.\]
\end{prop}

\begin{prop}\label{p3} Assume $n\geq3$  odd  and $\rm(H1)$  fulfilled. Then,  for all $b\geq\rho$ the eigenvalues  of $Z^b(T,0)$ are independent of the choice of $b$ and  for all $b\geq\rho$ there exists $T_1(b)$ such that for all $t,s\in{\R}$ satisfying  $t-s\geq T_1(b)$, $Z^b(t,s)$ is a compact  operator on ${\B}$.\end{prop}
\begin{prop}\label{p4}
Assume $n\geq3$  odd. Then, for all $0<\epsilon<b$ and all $\chi\in\mathcal{C}^\infty_0(\vert x\vert\leq b-\epsilon)$ we have
\begin{equation}\label{eq=prop3A}\chi P_b^\pm=P_b^\pm\chi=\chi.\end{equation}
\end{prop}

We refer to \cite{K1} and \cite{P1} for the proof of all these properties.

\begin{prop}\label{p6}
Let $\psi\in{\CI}$ be such that $\psi=1$ on $\vert x\vert\leq\rho+\half+T$. Then, we get
\begin{equation}\label{eq=prop4A}\mathcal U(T)-U_0(T)=\psi(\mathcal U(T)-U_0(T))=(\mathcal U(T)-U_0(T))\psi.\end{equation}
\end{prop}
\begin{proof}
Let $f\in{\B}$ and let $v$ be the function defined by $(v(t),v_t(t))=\mathcal U(t,0)(1-\psi)f$. The finite speed of propagation implies that, for all $0\leq t\leq T$ and $\vert x\vert\leq\rho+\half$, we have $v(t,x)=0$. Also, we find
\begin{equation}\label{eq=prop4B}\Delta_x=\Div_x(a(t,x)\nabla_x),\quad\textrm{for }\vert x\vert>\rho\end{equation}
and we deduce that $v$ is the solution, for $0\leq t\leq T$, of the problem
\[  \left\{\begin{array}{c}
v_{tt}-\Delta_xv=0,\\
(v,v_{t})(0,x)=(1-\psi(x))f(x).\end{array}\right.\]
It follows
\begin{equation}\label{eq=prop4C}(\mathcal U(T)-U_0(T))(1-\psi)=0.\end{equation}
Let $u$ and $v$ be functions defined by $(u(t),u_t(t))=\mathcal U(t,0)f$ and $(v(t),v_t(t))=U_0(t)f$ with\\
 $f\in{\B}$. Applying (\ref{eq=prop4B}), we deduce that $(1-\psi)u$ is a  solution of
\[  \left\{\begin{array}{c}
\partial_t^2((1-\psi)u)-\Delta_x(1-\psi)u=[\Delta_x,\psi]u,\\
((1-\psi)u,\partial_t((1-\psi)u)))(0,x)=(1-\psi(x))f(x)\end{array}\right.\]
and $(1-\psi)v$ is  a solution of the problem
\[  \left\{\begin{array}{c}
\partial_t^2((1-\psi)v)-\Delta_x(1-\psi)v=[\Delta_x,\psi]v,\\
((1-\psi)v,\partial_t((1-\psi)v)))(0,x)=(1-\psi(x))f(x).\end{array}\right.\]
Then, we have
\begin{equation}\label{eq=prop4D}(1-\psi)(\mathcal U(T)-U_0(T))=0.\end{equation}
Combining (\ref{eq=prop4C}) and (\ref{eq=prop4D}), we obtain (\ref{eq=prop4A}).
\end{proof}
\begin{Thm}\label{t2}

Assume $n\geq3$  odd and $\rm(H1)$  fulfilled. Let $\psi\in\mathcal{C}^\infty_0(\vert x\vert\leq\rho+T+1)$ be such that $\psi=1$  for $\vert x\vert\leq\rho+\half+T$. Let $\sigma(Z^\rho(T))$ be the spectrum of $Z^\rho(T,0)$.  Then the  eigenvalues   $\lambda \in \sigma(Z^\rho(T))\setminus \{0\}$ of $Z^\rho(T,0)$ coincide with the poles of $\psi(\mathcal U(T)-z)^{-1}\psi$.\end{Thm}
\begin{proof}
Following Proposition \ref{p3}, we  need to show this equivalence only for $Z^b(T)$ with $b\geq\rho$. Set $b=\rho+2+T$ and write $Z(T)$, $P_+$, $P_-$ instead of $Z^b(T)$,  $P^b_+$, $P^b_-$. In the same way, write $Z_0(T)$ instead of $Z_0^b(T)=P^b_+U_0(T)P^b_-$. Proposition 3  implies that the spectrum $\sigma(Z(T))$ of $Z(T)$ consists of eigenvalues and  $(Z(T)-z)^{-1}$ is meromorphic on $\mathbb C \setminus \{0\}$ (see also Chapter V of \cite{P1} ). For $\vert z\vert>\Vert\mathcal U(T)\Vert\geq\Vert Z(T)\Vert$, we have
\[\psi(Z(T)-z)^{-1}\psi=-\sum_{k=0}^\infty\frac{\psi(Z(T))^k\psi}{z^{k+1}}.\]
The properties of $Z^b(T)$ (see \cite{K1})  imply that, for $\vert z\vert>\Vert\mathcal U(T)\Vert\geq\Vert Z(T)\Vert$, we have
\[\psi(Z(T)-z)^{-1}\psi=-\sum_{k=0}^\infty\frac{\psi P_+(\mathcal U(kT))P_-\psi}{z^{k+1}}.\]
Since $b>\rho+T+1$ and $\psi\in\mathcal{C}^\infty_0(\vert x\vert\leq\rho+T+1)$,  applying (\ref{eq=prop1A})  and (\ref{eq=prop3A}), we get
\begin{equation}\label{eq=thm3A}\psi(Z(T)-z)^{-1}\psi=-\sum_{k=0}^\infty\frac{\psi (\mathcal U(T))^k\psi}{z^{k+1}}=\psi(\mathcal U(T)-z)^{-1}\psi.\end{equation}
Formula (\ref{eq=thm3A}) implies that  $\psi(\mathcal U(T)-z)^{-1}\psi$ is meromorphic on $\mathbb C \setminus \{0\}$ and the poles of $\psi(\mathcal U(T)-z)^{-1}\psi$ are included in the set   $\sigma(Z(T))\setminus\{0\}$. We will now show the inverse. Set
\[W(T)=Z_0(T)-Z(T)=P_+(U_0(T)-\mathcal U(T))P_-.\]
Applying (\ref{eq=prop3A}) and (\ref{eq=prop4A}), we deduce
\begin{equation}\label{eq=thm3B}W(T)=\psi V(T)\psi,\quad\textrm{with}\quad V(T)=U_0(T)-\mathcal U(T).\end{equation}
Next, let $z\in\mathbb C$ be such that $\vert z\vert>\Vert\mathcal U(T)\Vert$. We have
\[(Z(T)-z)^{-1}(Z_0(T)-Z(T))(Z_0(T)-z)^{-1}=(Z(T)-z)^{-1}-(Z_0(T)-z)^{-1},\]
and we get
\begin{equation}\label{eq=thm3C}(Z(T)-z)^{-1}=(Z(T)-z)^{-1}(Z_0(T)-Z(T))(Z_0(T)-z)^{-1}+(Z_0(T)-z)^{-1}.\end{equation}
Also, we obtain
\[(Z(T)-z)^{-1}=(Z_0(T)-z)^{-1}(Z_0(T)-Z(T))(Z(T)-z)^{-1}+(Z_0(T)-z)^{-1}\]
and applying  (\ref{eq=thm3C}) to the right-hand side of this equality, we conclude that

\[\begin{array}{lll}(Z(T)-z)^{-1}&=&(Z_0(T)-z)^{-1}(Z_0(T)-Z(T))(Z(T)-z)^{-1}(Z_0(T)-Z(T))(Z_0(T)-z)^{-1}\\
\ &\ &+(Z_0(T)-z)^{-1}(Z_0(T)-Z(T))(Z_0(T)-z)^{-1}+(Z_0(T)-z)^{-1}.\end{array}\]
Applying (\ref{eq=thm3B}), we find
\[\begin{array}{lll}
(Z(T)-z)^{-1}&=&(Z_0(T)-z)^{-1}\psi V(T)\psi(Z(T)-z)^{-1}\psi V(T)\psi(Z_0(T)-z)^{-1}\\
\ &\ &+(Z_0(T)-z)^{-1}\psi V(T)\psi(Z_0(T)-z)^{-1}+(Z_0(T)-z)^{-1}\end{array}\]
and (\ref{eq=thm3A}) implies
 \begin{equation}\label{eq=thm3D}\begin{array}{lll}
(Z(T)-z)^{-1}&=&(Z_0(T)-z)^{-1}\psi V(T)\psi(\mathcal U(T)-z)^{-1}\psi V(T)\psi(Z_0(T)-z)^{-1}\\
\ &\ &+(Z_0(T)-z)^{-1}\psi V(T)\psi(Z_0(T)-z)^{-1}+(Z_0(T)-z)^{-1}.\end{array}\end{equation}
The resolvent  $(Z_0(T)-z)^{-1}$ is holomorphic on $\mathbb C \setminus \{0\}$ and (\ref{eq=thm3D})  implies that all eigenvalues of $Z(T)$ different from  $0$ are poles of  $\psi(\mathcal U(T)-z)^{-1}\psi$. Thus, the resonances coincide with the poles of the meromorphic continuation of $\psi(\mathcal U(T)-z)^{-1}\psi$.\end{proof}
 Assuming (H1), Theorem \ref{t2} implies that assumptions (H2) and (H3) are equivalent, for $n\geq3$ odd.   
 \begin{rem}Combining the results of Theorem \ref{t2} with \cite{K1} we conclude that for $n\geq3$ odd we have \eqref{eq=locA} with $p(t)=e^{-\delta t}$, provided assumptions $\rm(H1)$ and $\rm(H2)$  fulfilled. Moreover, assuming $\rm(H1)$ and $\rm(H2)$, we obtain (\ref{eq=estC}) for $2\leq p,q<+\infty$ satisfying  (\ref{eq=stA}). \end{rem}

\section{Decay of  local energy for $n$ even}
Throughout  this section, we will show that the assumptions (H1) and (H2) imply for $n \geq 4$ the decay (\ref{eq=locA}) of the local energy. As a first step, we will show that we can generalize some results of Vainberg about the Fourier-Bloch-Gelfand transform of the propagator. Then, by applying these results, we will prove Theorem \ref{t3}.

\subsection{Assumptions and definitions}
In this subsection we introduce some notations and operators. We will also precise some assumptions. We follow closely the expositions in \cite{V2} and we present the corresponding results  for   problem (\ref{eq=lepbA}). 

 Assume
\begin{equation}\label{eq=VA}\mathcal U(t,s)=0\quad\textrm{for}\quad t<s \quad\textrm{and}\quad U_0(t)=0\quad\textrm{for}\quad t<0.\end{equation}
Let $P_1$ and $P_2$ be the projectors of $\mathbb C^2$ defined by \[  P_1(h)=h_1,\quad P_2(h)=h_2,\quad h=(h_1,h_2)\in\mathbb C^2\]
and let $P^1,P^2\in\mathcal L({\C},{\C}^2)$ be defined by
\[  P^1(h)=(h,0),\quad P^2(h)=(0,h),\quad h\in\mathbb C.\]

Denote by $V(t,s)$ the operator defined on $L^2({\R}^n)$ by
\[V(t,s)=P_1\mathcal U(t,s)P^2.\]
Notice that for $g\in L^2({\R}^n)$, $w=V(t,s)g$ is the solution of

\[  \left\{\begin{array}{c}
\partial_t^2(w)-\Div_{x}(a(t,x)\nabla_{x}w)=0,\\
(w,\partial_tw)_{\vert t=s}=(0,g).\end{array}\right.\]
Let $E(t,s,x,x_0)$ be the kernel of the operator $V(t,s)$. The propagation of singularities  (see \cite{H}) and (H1) imply that, for all $r>0$, there exists $T_1(r)$ such that 
\begin{equation}\label{eq=VB} E\in \mathcal C^\infty\quad\textrm{for}\quad\vert x\vert,\vert x_0\vert<r\ \textrm{and }t-s>T_1(r).\end{equation}
We may consider that $T_1(r)$ is a strictly increasing and regular function (see \cite{V2}). Moreover, we  assume that
\begin{equation}\label{eq=VC} T_1(r)=T_1(b),\quad r\leq b\quad\textrm{with } b=\rho+1+\frac{4}{5}+2T.\end{equation}
From now on, we set 
\[b=\rho+1+\frac{4}{5}+2T.\]
Let $T_2=T_2(r)$ be an increasing and regular function such that
\[ T_2(r)>T_1(r),\quad r>0\]
and
\[T_2(r)=k_0T,\quad r\leq b\quad\textrm{with } k_0\in\mathbb N.\]
Let $\xi(t,s,x)\in\mathcal C^\infty({\R}\times{\R}\times{\R}^n,{\R})$ be a function such that $\xi=0$ for $t-s>T_2(\vert x\vert)$, $\xi=1$ for $t-s\leq T_1(\vert x\vert)$, $\xi=\xi(t-s)$ for $\vert x\vert\leq b$. Let $\psi\in\mathcal C^\infty({\R}^n_x)$ be a function such that $\psi=1$ for $\vert x\vert\geq b-\frac{1}{3}$ and $\psi=0$ for $\vert x\vert\leq b-\frac{2}{3}$. Denote by $P(t)$ the differential operator
\[P(t)=\partial_t^2-\Div_x(a(t,x)\nabla_x).\]
Set $\Lambda=\sqrt{-\Delta_x}$ and let  $W(t,s)$ be the operator
\begin{equation}\label{eq=VD}W(t,s)=\xi(t,s)V(t,s)-\psi N(t,s),\end{equation}
with
\[N(t,s)=\int_s^t\frac{\sin(\Lambda (t-\tau))}{\Lambda}\psi[P(\tau),\xi(\tau,s)]V(\tau,s)d\tau.\]
Let $h\in L^2({\R}^n)$. Then $w_1=N(t,s)h$ is the solution of
\[  \left\{\begin{array}{c}
\partial_t^2w_1-\Delta_xw_1=\psi[P(t),\xi(t,s)]V(t,s)h,\\
(w_1,\partial_tw_1)_{\vert t=s}=(0,0).\end{array}\right.\]
From (\ref{eq=prop4B}), we deduce that   $w_2=\psi N(t,s)h$ is the solution of
\[  \left\{\begin{array}{c}
\partial_t^2w_2-\Div_x(a(t,x)\nabla_xw_2)=-[\Delta_x,\psi]N(t,s)h+\psi^2[P(t),\xi(t,s)]V(t,s)h,\\
(w_2,\partial_tw_2)_{\vert t=s}=(0,0).\end{array}\right.\]
It follows that  $w_3=W(t,s)h$ is the solution of
\[  \left\{\begin{array}{c}
\partial_t^2w_3-\Div_x(a(t,x)\nabla_xw_3)=G(t,s)h,\\
(w_3,\partial_tw_3)_{\vert t=s}=(0,h)\end{array}\right.\]
with
\[ G(t,s)h=[\Delta_x,\psi]N(t,s)h+(1-\psi^2)[P(t),\xi(t,s)]V(t,s)h.\]
\begin{prop}\label{p7}
Let $\chi\in{\CII}$. Then, the operator $G(t,s)\chi$ is a compact operator of $L^2({\R}^n)$.
\end{prop}
\begin{proof}
Let $\chi\in{\CII}$. The properties of $\xi$ implies
\begin{equation}\label{eq=prop8A}[P(t),\xi(t,s,x)]=0,\quad\textrm{for }t-s<T_1(\vert x\vert)\textrm{ or for }t-s>T_2(\vert x\vert).\end{equation}
Since $1-\psi^2(x)=0$ for $\vert x\vert\geq b-\frac{1}{3}$, the properties (\ref{eq=VB}) and (\ref{eq=prop8A}) imply that $(1-\psi^2)[P(t),\xi(t,s)]V(t,s)\chi$ 
is a compact operator of $L^2({\R}^n)$.
Therefore (\ref{eq=VC}) and (\ref{eq=prop8A}) imply that the kernel  $N(t,s,x,x_0)$ of the operator $N(t,s)$ satisfies
\begin{equation}\label{1111aaaa} N(t,s,x,x_0)=0,\ \textrm{for }t-s< T_1(b).\end{equation}
Applying (\ref{eq=VB}) and (\ref{eq=VC}), we obtain
\begin{equation}\label{eq=prop8B}N(t,s,x,x_0)\in\mathcal C^\infty,\ \textrm{for } \vert x\vert,\vert x_0\vert\leq b.\end{equation}
Then, since  $[\Delta_x,\psi](x)=0$ for $\vert x\vert>b$, (\ref{eq=prop8B}) implies that $[\Delta_x,\psi]N(t,s)\chi$ is a compact operator in $L^2({\R}^n)$. We deduce that $G(t,s)\chi$ is compact operator of $L^2({\R}^n)$.\end{proof}

To prove (\ref{eq=locA}), we will use some results established in \cite{V2} for $s=0$. For our purpose we need to consider the case $0\leq s<T$. The proofs are similar and we will only give a proof when it is necessary, otherwise we refer to \cite{V2}. 

\begin{Thm}\label{t4} \emph{(Theorem 1, \cite{V2})}\\
1) For all $h\in{\CII}$ and $s\geq0$, the integral equation
\begin{equation}\label{eq=thm5A}\phi_s(t,.)+\int_s^tG(t,\tau)\phi_s(\tau,.)d\tau=-G(t,s)h,\end{equation}
admits an unique solution $\phi_s(t,x)$, with $\phi_s\in \mathcal C^\infty({\R}\times{\R}^n)$ such that\\
 $\textrm{supp}_x\phi_s\subset\{x\ :\ \vert x\vert\leq b\}$ and $\phi_s(t,x)=0$ for $t\leq s+T_1(b)$.\\
2) For all $h\in{\CII}$ and for  the solution $\phi_s$ of \eqref{eq=thm5A} we have
\begin{equation}\label{eq=thm5B}V(t,s)h=W(t,s)h+\int_s^tW(t,\tau)\phi_s(\tau,.)d\tau.\end{equation}
\end{Thm}

Let $r\in{\R}$. We denote  $H^{r,A_1}_{b,s}({\R}^{1+n})$ the space defined by $g\in H^{r,A_1}_{b,s}({\R}^{1+n})$ if
\[\begin{array}{l}(i)\ e^{-A_1t}\phi\in H^{r}({\R}^{1+n}),\\
(ii)\ g(t,x)=0\ \textrm{for }t\leq s\ \textrm{or }\vert x\vert\geq b.\end{array}\]
 The spaces $H^r_b({\R}^n)$, $H^r_b({\R}^{1+n})$, $\mathcal C^\infty_b({\R}^{n})$ and $\mathcal C^\infty_b({\R}^{1+n})$ are the subspaces of $H^r({\R}^n)$, $H^r({\R}^{1+n})$, $\mathcal C^\infty({\R}^{n})$ and $\mathcal C^\infty({\R}^{1+n})$, respectively, consisting of functions that vanish for $\vert x\vert\geq b$. The global energy estimate  (\ref{eq=prop2A}) implies that, for $A_1>A$ (with $A$ the constant in  (\ref{eq=prop2A})) and for $\chi\in{\CII}$, we have \begin{equation}\label{1}\chi V(t,s)\in\mathcal L(L^2_b,H^{1,A_1}_{b,s}({\R}^n)).\end{equation}
 Throughout  this section we assume  $A_1>A$.
\subsection{Properties of the  Fourier-Bloch-Gelfand transform on the spaces $H^{r,A_1}_{b,s}({\R}^{1+n})$ for $s\geq0$}
In this subsection we will recall some properties of the  Fourier-Bloch-Gelfand transform on the spaces  $H^{r,A_1}_{b,s}({\R}^{1+n})$. Vainberg established  in \cite{V2}  these results for $s=0$. All these properties hold for   $s>0$.
We denote by $F$, the Fourier-Bloch-Gelfand transform defined on $H^{r,A_1}_{b,s}({\R}^{1+n})$ such that, for $\textrm{Im}(\theta)\geq A_1T$, we have
\[F(\phi)(t,\theta)=\sum_{k=-\infty}^{+\infty}\phi(kT+t)e^{ik\theta},\quad\phi\in{\BI}.\]
Let $\phi\in H^{r,A_1}_{b,s}({\R}^n)$. For $ \textrm{Im} (\theta)>A_1T$, we write
\[\hat{\phi}(t,\theta)=F(\phi)(t,\theta).\]
\begin{prop}\label{p8}\emph{(Lemma 1, \cite{V2})}
Let $\phi\in H^{r,A_1}_{b,s}({\R}^{1+n})$. Then,  for $ \textrm{Im} (\theta)>A_1T$ the following assertions hold:\\
1) for any $B>0$ the operator
\[F\ :\ H^{r,A_1}_{b,s}({\R}^{1+n})\rightarrow H^r_b([s-B,s+B]\times {\R}^n)\]
is bounded and  depends analytically on $\theta$.\\
2)\[\hat{\phi}(t,\theta+2\pi)=\hat{\phi}(t,\theta);\]
3)\[\hat{\phi}(t+T,\theta)=e^{-i\theta}\hat{\phi}(t,\theta),\]
and hence if $v(t,\theta)=e^{\frac{it\theta}{T}}\hat{\phi}(t+T,\theta)$, 
$v(t+T,\theta)=v(t,\theta)$.\\
4)Let $\alpha> A_1T$, $c\in{\R}$. If  $d_{\alpha,c}$ is the interval $[i\alpha+c,i\alpha+c+2\pi]$, then
\begin{equation}\label{eq=VF}\phi(t)=\frac{1}{2\pi}\int_{d_{\alpha,c}}F(\phi)(t,\theta)\d\theta,\quad t\in{\R}.\end{equation}
\end{prop}
Denote by $H^{r}_{b,s,per}({\R}^{1+n})$ the norm closure of the subspace of   $ H^{r}([s,s+T]\times{\R}^n)$ consisting of the infinitely differentiable functions in $[s,s+T]\times{\R}^n$ which are $T$-periodic with respect to $t$ and vanish for $\vert x\vert\geq b$. Let $F'$ be the operator defined by
\[F'(\phi)(t,\theta)=e^{\frac{i\theta t}{T}}F(\phi)(t,\theta).\]
The following result is a trivial consequence of  Proposition \ref{p8}.
\begin{prop}\label{p9}
For $ \textrm{Im}( \theta)>A_1T$ the operator
\[F'\ :\ H^{r,A_1}_{b,s}({\R}^{1+n})\rightarrow H^r_{b,s,per}({\R}^{1+n})\]
is bounded and  depends analytically on $\theta$.\end{prop}
Now, we will consider the  Fourier-Bloch Gelfand transformation of operators with a kernel.
\begin{prop}\label{p10}\emph{(Lemma 4, \cite{V2})}
Suppose that the operator
\[ R\ :\ H^{r,A_1}_{b,s}({\R}^{1+n})\rightarrow H^{l,A_1}_{b,s}({\R}^{1+n})\]
is bounded and its kernel has the properties: \\
1) $ R(t+T,\tau+T,x,x_0)=R(t,\tau,x,x_0)$.\\
2) There exists a $T_0>0$ such that
\begin{equation}\label{prop27b}R(t,\tau,x,x_0)=0,\quad\textrm{for}\quad t-\tau\notin [0,T_0].\end{equation}
Then, there exists an operator
\[R(t,s,\theta)\ :\ H^{r}_{b,s,per}({\R}^{1+n})\rightarrow H^{l}_{b,s,per}({\R}^{1+n}),\]
such that $R(t,s,\theta)$ is an entire function on $\theta$ and
$F'(R)=R(t,s,\theta)F'$ for\\
 $ \textrm{Im}(\theta)>A_1T$.
\end{prop}
We establish easily the following result.
\begin{prop}\label{p11}
Suppose that the operator
\[R(t,s)\ :\ H^r_b({\R}^n)\rightarrow H^{l,A_1}_{b,s}({\R}^{1+n})\]
is bounded. Then, the operator
\[R(t,s,\theta)=F'(R(t,s))(t,\theta)\ :\ H^r_b({\R}^n)\rightarrow H^{l}_{b,per}({\R}^{1+n})\]
is bounded and $R(t,s,\theta)$ is an entire function on $\theta$ for $ \textrm{Im}(\theta)>A_1T$.\end{prop}
Set $\psi_1,\psi_2\in{\CII}$. We will analyze the properties of the Fourier-Bloch-Gelfand transform of $\psi_1 V(t,s)\psi_2$.
\subsection{  Fourier-Bloch-Gelfand transform of the operator  $\psi_1 V(t,s)\psi_2$}
In this subsection, our aim is to analyze the composition of $F'$  with the operator $\psi_1 V(t,s)\psi_2$. More precisely, we will show that, for $0\leq s<T$ and $t\geq T_2(b)+T$, $F'(\psi_1 V(t,s)\psi_2)(t,\theta)$, initially defined for $ \textrm{Im}(\theta)>A_1T$, admits a meromorphic continuation satisfying some properties that we will precise. From these results we will establish the asymptotic behavior as $t\rightarrow+\infty$ of the local energy associated to (\ref{eq=lepbA}).  In \cite{V2}, Vainberg  proves the properties of $F'(\psi_1 V(t,s)\psi_2)(t,\theta)$, when  $s=0$ and $t\geq T_2(b)$. But all these results remain true when $0\leq s<T$ and $t\geq T_2(b)+T$.

Denote $\mathbb C'=\{ z\in\mathbb C\ :\ z\neq 2k\pi-i\mu,\ k\in\mathbb{Z},\ \mu\geq0\}$.\\
\begin{defi}Let $H_1$ and $H_2$ be Hilbert spaces. A family of bounded operators
 $Q(t,s,\theta):H_1\rightarrow H_2$ is said to be meromorphic in a domain $D\subset\mathbb C$, if $Q(t,s,\theta)$ is meromorphically dependent on $\theta$ for $\theta\in D$ and for any pole  $\theta=\theta_0$  the coefficients of the negative powers of  $\theta-\theta_0$ in the appropriate Laurent extension are finite-dimensional operator. \end{defi}
 
\begin{defi} We say that the family of operators $Q(t,s,\theta)$, which are $\mathcal C^\infty$ and $T$-periodic with respect to $t$, has the property $(S)$ if: 1) when $n$ is odd, the operators $Q(t,s,\theta)$, $\theta\in\mathbb C$ and its derivatives with respect to $t$ are bounded and form a finitely meromorphic family ; 2) When $n$ is even the operators $Q(t,s,\theta)$ and its derivatives with respect to $t$ are bounded, finitely meromorphic on $\theta$ for $\theta\in\mathbb{C}'$ and, in a neighborhood of $\theta=0$, $Q(t,s,\theta)$ has the following form 
 \begin{equation}\label{asymp}Q(t,s,\theta)=B(t,s,\theta)\log\theta+\sum_{j=1}^{m}B_j(t,s)\theta^{-j}+C(t,s,\theta),\end{equation}
where the operators $B(t,s,\theta)$ and $C(t,s,\theta)$  depend analytically on  $\theta$ for $\vert\theta\vert<\epsilon_0$,  $\log$ is the logarithm defined on  $\mathbb{C}\setminus i{\R}^-$, and for all $j$ the operators $B_j(t,s)$ and $(\partial_\theta B(t,s,\theta))_{\vert\theta=0}$ are finite dimensional. Moreover, $B(t,s,\theta)$, $C(t,s,\theta)$  and the the operators $B_j(t,s)$ are 
$\mathcal C^\infty$ and $T$-periodic with respect to $t$  and depend on  $s$.\end{defi}

Denote by $G_s$ and $W_s$ the operators defined for all $\phi\in H^{1,A_1}_{b,s}({\R}^{1+n})$, by 
\[G_s(\phi)(t)=\int_s^tG(t,\tau)\phi(\tau)d\tau,\quad W_s(\phi)(t)=\int_s^tW(t,\tau)\phi(\tau)d\tau.\]
We recall some results about the properties of the composition of $F'$ and the operators $G_s$, $G(t,s)$, $\chi W_s$ and $\chi N(t,s)$, with $\chi\in{\CII}$. 

\begin{Thm}\label{t5}\emph{(Theorem 2, \cite{V2})} Let $0\leq s<T$. The operator
\begin{equation}\label{thm}G_s\ :\ H^{1,A_1}_{b,s}({\R}^{1+n})\rightarrow H^{2,A_1}_{b,s}({\R}^{1+n})\end{equation}
is bounded, and for $ \textrm{Im}(\theta)>A_1T$ the relation $F'(G_s)(t,\theta)=G_s(t,s,\theta)F'$ holds, where 
\[G_s(t,s,\theta)\ :\ H^{1}_{b,s,per}({\R}^{1+n})\rightarrow H^{2}_{b,s,per}({\R}^{1+n})\]
is an operator with the property $(S)$.\end{Thm}

\begin{Thm}\label{t6}\emph{(Theorem 3, \cite{V2})} Let $0\leq s<T$. For all $c>0$ and for all $r\in{\R}$, the operator
\begin{equation}\label{thm}G(t,s)\ :\ H^{1}_{b}({\R}^{1+n})\rightarrow H^{r,c}_{b,s}({\R}^{1+n})\end{equation}
is bounded and the operator
\[G(t,s,\theta)=F'(G(t,s))(t,\theta)\ :\ H^{1}_{b}({\R}^{1+n})\rightarrow H^{r}_{b,s,per}({\R}^{1+n})\]
defined  for  $ \textrm{Im}(\theta)>A_1T$, has an analytic continuation to the lower half plane with the property $(S)$.
\end{Thm}

\begin{Thm}\label{t7}\emph{(Theorem 4 and Lemma 8, \cite{V2})} Let $\chi\in{\CII}$ and $0\leq s<T$. The operator
\begin{equation}\label{thm1}\chi W_s\ :\ H^{1,A_1}_{b,s}({\R}^{1+n})\rightarrow H^{0,A_1}_{b,s}({\R}^{1+n}),\end{equation}
is bounded, and for $ \textrm{Im}(\theta)>A_1T$ the relation $F'(W_s)(t,\theta)=W_s(t,s,\theta)F'$ holds, where 
\[\chi W_s(t,s,\theta)\ :\ H^{1}_{b,s,per}({\R}^{1+n})\rightarrow H^{2}_{b,s,per}({\R}^{1+n}),\]
is an operator with the property $(S)$. The operator
\begin{equation}\label{thm}\chi N(t,s)\ :\ H^{1}_{b}({\R}^{1+n})\rightarrow H^{2,A_1}_{b,s}({\R}^{1+n}),\end{equation}
is bounded and the operator
\[\chi N(t,s,\theta)=F'(\chi N(t,s))(t,\theta)\ :\ H^{1}_{b}({\R}^{1+n})\rightarrow H^{2}_{b,s,per}({\R}^{1+n})\]
 defined for $ \textrm{Im}(\theta)>A_1T$, admits an analytic continuation with property $(S)$.
\end{Thm}

\begin{defi} We say that the family of operators $Q(t,s,\theta)$, which are $\mathcal C^\infty$ and $T$-periodic with respect to $t$, has the property $(S')$ if:
1) for odd $n$ the operators $Q(t,s,\theta),\ \theta\in\mathbb{C},$ and its derivatives with respect to $t$ form a finitely-meromorphic family; 2) For even $n$ the operators $Q(t,s,\theta)$ and its derivatives with respect to $t$ form a finitely-meromorphic family for $\theta\in\mathbb{C}'$ . Moreover, in a neighborhood of $\theta=0$ in $\mathbb C'$, $Q(t,s,\theta)$ has the form
\begin{equation}\label{eq=VG}Q(t,s,\theta)=\theta^{-m}\sum_{j\geq0}\left(\frac{\theta}{R_{t,s}(\log \theta)}\right)^jP_{j,t,s}(\log \theta)+C(t,s,\theta),\end{equation}
where $C(t,s,\theta)$ is analytic with respect to $\theta$,  $R_{t,s}$ is a polynomial, the $P_{j,t,s}$ are polynomials of order at most $l_j$ and $\log$ is the logarithm defined on $\mathbb{C}\setminus i{\R}^-$. Moreover, $C(t,s,\theta)$ and the coefficients of  the polynomials $R_{t,s}$ and $P_{j,t,s}$ are $\mathcal C^\infty$ and $T$-periodic with respect to $t$ and depend of $s$.\end{defi}
\begin{rem} Notice that if $Q(t,s,\theta)$ satisfies (S') then $\partial_tQ(t,s,\theta)$ satisfies also (S').\end{rem}

\begin{Thm}\label{t8}\emph{(Theorem 5, \cite{V2})} Let $B>A$ with $A$ the constant of estimate  \eqref{eq=prop2A}. Then, there exists $A_2>B$ such that for all $h\in H^{1,B}_{b,s}({\R}^{1+n})$ with $0\leq s<T$, the equation
\begin{equation}\label{11}\phi+\int_s^tG(t,\tau)\phi(\tau)d\tau=h,\end{equation}
is uniquely solvable in the space $H^{1,A_1}_{b,s}({\R}^{1+n})$ for any $A_1\geq A_2$,  and
\begin{equation}\label{12}\Vert \phi\Vert_{H^{1,A_1}_{b,s}({\R}^{1+n})}\leq C(s)\Vert h\Vert_{H^{1,B}_{b,s}({\R}^{1+n})}.\end{equation}\end{Thm}

The next result is a trivial consequence of Theorems \ref{t4}, \ref{t6} and \ref{t8}.
\begin{prop}\label{p12}  Let $0\leq s<T$ and $A_1\geq A_2$, with $A_2$ the constant of Theorem \ref{t8} for a $B>A$.
Then, there exists an operator \[L(t,s)\ :\ L^2_b\rightarrow H^{1,A_1}_{b,s}({\R}^{1+n})\] such that $L(t,s)$ is bounded and satisfies
\begin{equation}\label{eq=prop9A}L(t,s)h+\int_s^tG(t,\tau)L(\tau,s)hd\tau=-G(t,s)h,\quad h\in L^2_b,\  t\geq T.\end{equation}
\end{prop}
In the following, we assume that $A_1\geq A_2$ with $A_2$ the constant of Theorem \ref{t8} for a $B>A$.
We will now recall a result, established by  Vainberg, which will allow us to define the properties of the  Fourier-Bloch-Gelfand transform of $L(t,s)$.
\begin{Thm}\label{t9} \emph{(\cite{V2}, Theorem 9)} Let $H$ be an Hilbert space and let $G(t,s,\theta)\ :\ H\rightarrow H$ be a family of compact operators having the property $(S)$. If there exists $\theta_0$ such that $Id+G(t,s,\theta_0)$ is invertible, then the family of operators $(Id+G(t,s,\theta))^{-1}$ has the property $(S')$.\end{Thm}
 Applying this result, we deduce the following.
 \begin{Thm}\label{t10}
Let $0\leq s<T$. The operator
 \[L(t,s,\theta)\ :\ L^2_b\rightarrow H^1_{b,s,per}({\R}^{n+1}),\]
defined originally  for $ \textrm{Im}(\theta)>A_1T$ by the relation
 \[L(t,s,\theta)=F'(L(t,s))(t,\theta),\]
 admits an analytic continuation having the property $(S')$.\end{Thm}
 \begin{proof}
We apply the operator $F'$ to both sides of (\ref{eq=prop9A}). It follows from Theorems \ref{t6} and \ref{t7}, and Propositions \ref{p11} and \ref{p12}, that, for $ \textrm{Im}(\theta)>A_1T$, $F'(L(t,s))(t,\theta)$ satisfies
 \begin{equation}\label{thm24a} (Id+G_s(t,s,\theta))F'(L(t,s))(t,\theta)=-G(t,s,\theta).\end{equation}
We consider the operator $G_s(t,s,\theta)$ acting in the spaces
 \begin{equation}\label{thm24b}G_s(t,s,\theta)\ :\  H^1_{b,s,per}({\R}^{1+n})\rightarrow H^1_{b,s,per}({\R}^{1+n}).\end{equation}
It follows from Theorem \ref{t5} that (\ref{thm24b}) is compact. Consequently, we deduce from (\ref{thm24a}), Theorem \ref{t10} and the properties of  operators $G_s(t,s,\theta)$ and $G(t,s,\theta)$, established in Theorems \ref{t5} and \ref{t6}, that Theorem \ref{t10} is valid if we  show that there exists  $D>A_1T$ such that for $\theta=iD$ the operator $(Id+G_s(t,s,\theta))$ is invertible. To prove the latter it clearly suffices to show that for some $D>A_1T$ and for $\theta=iD$ the equation
\begin{equation}\label{thm24c}(Id+G_s(t,s,\theta))\psi=\phi,\quad \phi,\psi\in H^1_{b,s,per}({\R}^{1+n}),\end{equation}
 is solvable for all $\phi$. Let $g\in  H^1_{b,s,per}({\R}^{1+n})$, and $\gamma\in\mathcal C^\infty({\R})$ be such that $0\leq\gamma\leq1$, $\gamma(t)=0$ for $t\leq s+\frac{T}{2}$, $\gamma(t)=1$ for $t\geq s+\frac{2T}{3}$. We see from Theorem \ref{t8} that the equation
 \[\phi_1+\int_s^tG(t,\tau)\phi_1(\tau)d\tau=\gamma g,\]
has a unique solution $\phi_1\in H^{1,A_1}_{b,s}({\R}^{1+n})$. Theorem \ref{t5} implies that for $ \textrm{Im}(\theta)>A_1T$, the equation (\ref{thm24c}) has a unique solution
 \[\psi=F'(\phi_1)\in H^1_{b,s,per}({\R}^{1+n})\quad \textrm{for}\quad \phi=F'(\gamma g).\]
 Set $D>A_1T$. For the proof of the theorem it suffices to show that for any $\phi\in H_{b,s,per}^1({\R}^{1+n})$, we can choose $g\in H^1_{b,s,per}({\R}^{1+n})$ such that
 \begin{equation}\label{thm24d}\phi=[F'(\gamma g)]_{|\theta=iD}.\end{equation}
For $ \textrm{Im}(\theta)>A_1T$ and $t\in[s,s+T]$, we have
 \[\begin{array}{lll}F'(\gamma g)(t,\theta)&=&e^{\frac{i\theta t}{T}}\sum_{k=0}^{+\infty}(\gamma g)(kT+t)e^{ik\theta}\\
 \ &=&e^{\frac{i\theta t}{T}}\left((\gamma g)(t)+\sum_{k=0}^{+\infty}g(kT+t)e^{ik\theta}\right)\\
 \ \\
 \ &=&e^{\frac{i\theta t}{T}}\left((\gamma g)(t)+g(t)\sum_{k=0}^{+\infty}e^{ik\theta}\right)\\
 \ \\
 \ &=&e^{-\frac{D t}{T}}g(t)\left[\gamma (t)+(1-e^{-D})^{-1}e^{-D}\right].\end{array}\]
 Let $p_1$ be a function defined on $s\leq t\leq s+T$, by
 \[p_1(t)=e^{-\frac{Dt}{T}}\left[\gamma(t)+(1-e^{-D})^{-1}e^{-D}\right].\]
 For all $s\leq t\leq s+\frac{T}{2}$, we have
  \[p_1(t)=e^{-\frac{Dt}{T}}(1-e^{-D})^{-1}e^{-D}\]
  and, for all $s+\frac{3T}{2}\leq t\leq s+T$, we get
  \[\begin{array}{lll}p_1(t)&=&e^{-\left(\frac{D(t-T)}{T}\right)}e^{-D}\left[1+(1-e^{-D})^{-1}e^{-D}\right]\\
  \ &=&e^{-\left(\frac{D(t-T)}{T}\right)}\left[e^{-D}+(1-e^{-D})^{-1}e^{-2D}\right]\\
  \ &=&e^{-\left(\frac{D(t-T)}{T}\right)}\left(1-e^{-D}\right)^{-1}e^{-D}.\end{array}\]
Thus, for all $N\in\mathbb N$, we obtain
  \[\frac{d^Np_1}{\d t^N}(s)=\frac{d^Np_1}{\d t^N}(s+T).\]
  Consequently, we can define a function    $p\in\mathcal C^\infty({\R})$ and $T$-periodic such that
  \[p(t)=p_1(t),\quad t\in[s,s+T].\]
Since $\gamma(t)\geq0$, it follows that $p(t)>0$ for all $t\in{\R}$. Then, for any   $\phi\in H_{b,s,per}^1({\R}^{1+n})$, we have (\ref{thm24d}) if
 \[g(t,.)=\frac{\phi(t,.)}{p(t)}.\]
 \end{proof}

Denote by $R(t,s)$ the operator defined by
\begin{equation}\label{eq=VH}R(t,s)=-\psi N(t,s)+\int_s^tW(t,\tau)L(\tau,s)d\tau.\end{equation}
Then, we can extend the result established by Vainberg for $s=0$, in the following way.
\begin{Thm}\label{t11} 
Let $0\leq s<T$ and let $\chi\in{\CII}$. The operator
\[\chi R(t,s):\ L^2_b\rightarrow {\BI}\]
is bounded. Moreover, the family of operators
\[ R(t,s,\theta):\ L^2_b\rightarrow H^1_{b,per}({\R}^{1+n}),\quad R(t,s,\theta)=F'(\chi R(t,s))(t,\theta)\]
 defined for $ \textrm{Im}(\theta)>A_1T$, admits an analytic continuation to the lower half plane and this continuation has the property $(S')$.\end{Thm}
Vainberg established the result of Theorem \ref{t11}, in the Theorem 11 of \cite{V2} for $s=0$ and $t> T_2(b)$. Combining this result with Theorems \ref{t7} and \ref{t10}, and with the estimate (\ref{1}), we see that this result holds for $0\leq s<T$ and $t\geq T_2(b)+T$.  
\begin{rem}
Notice that \emph{Theorem \ref{t11}} does not give any information about the dependence of $R(t,s,\theta)$ with respect to $s$.
\end{rem}
Combining the representations (\ref{eq=thm5B}) and (\ref{eq=prop9A}), and applying an argument of density, we get
\begin{equation}\label{eq=VE}V(t,s)h=W(t,s)h+\int_s^tW(t,\tau)L(\tau,s)h\ d\tau,\quad h\in L^2_b.\end{equation}
The properties of $\xi$, for  $t-s>T_2(b)$, imply
\[\chi\xi(t,s)=0,\quad \chi\in\mathcal C^\infty_0(\vert x\vert\leq b).\]
Combining this with the formulas (\ref{eq=VD}), (\ref{eq=VE}) and (\ref{eq=VH}),  for $0\leq s<T$ and $t\geq T_2(b)+T$, we find
\[\chi_1V(t,s)\chi_2=\chi_1R(t,s)\chi_2,\quad\chi_1,\chi_2\in{\CII}.\]
Theorem \ref{t11} implies that, for  $0\leq s<T$ and $t\geq T_2(b)+T$, $F'(\chi_1V(t,s)\chi_2)(t,\theta)$ admits an analytic continuation to the lower half plane with the property $(S')$. This result together with the assumption (H2) will be combined to establish (\ref{eq=thm2A}) for even dimensions.\\
\subsection{ Proof of Theorem \ref{t3}}
The goal of this subsection is to prove  Theorem \ref{t3}. From now on, let $\chi_j,\psi_j\in\mathcal C_0^\infty(\vert x\vert<\rho+1+\frac{j+1}{5}+(j-1)T)$, $j\in\{1,\ldots,4\}$, be such that for all $j\in\{2,3,4\}$,
we have
\begin{equation}\label{dudu}\psi_{j}(x)=\chi_{j}(x)=1, \quad \textrm{ for }\vert x\vert\leq\rho+1+\frac{j}{5}+(j-1)T.\end{equation}
Notice that, for all $j\in\{1,2,3\}$, we obtain
\[\chi_{j+1}=1\ \textrm{ on }\textrm{supp}(\chi_j)+T,\quad \psi_{j+1}=1\ \textrm{ on }\textrm{supp}(\psi_j)+T.\]
Consider $V(t,s,\theta)=F'(V(t,s))(t,\theta)$.
In  subsections 3.1, 3.2 and 3.3 we have generalized the results of \cite{V2} and proved that $V(t,s,\theta)$ satisfies property $(S')$ for $0\leq s<T$ and $t\geq T_2(b)+T$. Following \cite{V2}, we can establish the asymptotic behavior as  $t\to+\infty$ of $\chi_3V(t,s)\psi_3$. Nevertheless, we cannot deduce directly (\ref{eq=thm2A}).
To prove (\ref{eq=locA}), we  establish a link between $R_{\chi_4,\psi_4}(\theta)$ and $V(t,s,\theta)$, and we show how $\rm(H3)$ is related to the meromorphic continuation of $V(t,s,\theta)$. Then, applying the results of \cite{V2}, for $t\geq(k_0+1)T$ and $0\leq s\leq\frac{2T}{3}$, we obtain
\begin{equation}\label{nnn}\norm{\chi_3V(t,s)\psi_3}\leq \frac{C}{(t+1)\ln^2(t+e)},\quad\norm{\chi_1\partial_tV(t,s)\chi_1}\leq \frac{C}{(t+1)\ln^2(t+e)}\end{equation}
 with $C$ independent of $s$ and $t$. Consider the operator defined by
 \[U(t,s)=P_1\mathcal U(t,s)P^1.\] For all $h\in\dot{H}^1({\R}^n)$, $w=U(t,s)h$ is the solution of
\[  \left\{\begin{array}{c}
\partial_t^2w-\Div_{x}(a(t,x)\nabla_{x}w)=0,\\
(w,w_t)_{\vert t=s}=(h,0).\end{array}\right.\]
If $a(t,x)$ is independent of $t$, we have
 \begin{equation}\label{nn}\partial_tV(t,s)f-V(t,s)\left((\partial_t^2V(t,s)f)_{|t=s}\right)=U(t,s)f,\quad f\in{\CI}\end{equation}
 and \eqref{eq=thm2A} follows easily  from \eqref{nnn}. If $a(t,x)$ is time-dependent,  statement \eqref{nn} is not true and it will be more difficult to prove that \eqref{nnn} implies \eqref{eq=thm2A}.

To prove \eqref{eq=thm2A}, we start by showing the link between $F'(\chi_3V(t,s)\psi_3)(t,\theta)$ and $R_{\chi_4,\psi_4}(\theta)$.
\begin{lem}\label{l1} Assume $\rm(H1)$ and $\rm(H2)$ are fulfilled and let $n\geq4$ be even. Let  $t\geq(k_0+1)T$, $0\leq s\leq\frac{2T}{3}$. Then, the family of operators $V(t,s,\theta)=F'(\chi_3V(t,s)\psi_3)(t,\theta)$ admits an analytic continuation to $\{ \theta\in\mathbb C'\ :\ \textrm{Im}(\theta)\geq0\}$ and we have \begin{equation}\label{eq=thm2E}\displaystyle\limsup_{\substack{\lambda\to0 \\ \im(\lambda)>0}}\left(\sup_{s\in[0,\frac{2T}{3}]}\Vert V(t,s,\lambda)\Vert_{\mathcal L(L^2({\R}^n),\dot{H}^1({\R}^n))}\right)<\infty.\end{equation} \end{lem}
\begin{proof} Notice that, from (2.2), for $\im(\theta)>AT$ and for all $\phi_1,\phi_2\in{\CI}$, we have
\begin{equation}\label{lalali}R_{\phi_1,\phi_2}(\theta)=-e^{i\theta}\sum_{k=0}^\infty\phi_1\mathcal U(kT)\phi_2e^{ik\theta}.\end{equation}
Set $k_2\in\mathbb N$ such that $0\leq t'=t-k_2T<T$. Assume $t'\geq s$. Then, for $\im(\theta)>AT$, we find
\begin{equation}\label{eq=thm2V}\begin{array}{lll}F'(\chi_3V(t,s)\psi_3)(t,\theta)&=&F'(P_1\chi_3\mathcal U(t,s)\psi_3P^2)(t,\theta)\\
\ &=&\displaystyle e^{i\frac{t}{T}\theta}\left(\sum_{k=-k_2}^\infty P_1\chi_3\mathcal U(t+kT,s)\psi_3P^2e^{ik\theta}\right)\\
\ \\

\ &=&\displaystyle P_1\left(e^{i\frac{t}{T}\theta}\sum_{k=-k_2}^\infty \chi_3\mathcal U(t+kT,s)\psi_3e^{ik\theta}\right)P^2
.\end{array}\end{equation}
Moreover, we obtain 
\[e^{i\frac{t}{T}\theta}\sum_{k=-k_2}^\infty \chi_3\mathcal U(t+kT,s)\psi_3e^{ik\theta}=e^{i\left(\frac{t}{T}-k_2\right)\theta}\chi_3\mathcal U(t',s)\psi_3+e^{i\frac{t}{T}\theta}\sum_{k=-(k_2-1)}^\infty \chi_3\mathcal U(t+kT,s)\psi_3e^{ik\theta}\]
and since $\frac{t}{T}=k_2+\frac{t'}{T}$, we get
\[e^{i\frac{t}{T}\theta}\sum_{k=-k_2}^\infty \chi_3\mathcal U(t+kT,s)\psi_3e^{ik\theta}=e^{i\frac{t'}{T}\theta}\chi_3\mathcal U(t',s)\psi_3+e^{i\frac{t'}{T}\theta}\left(e^{ik_2\theta}\sum_{k=-(k_2-1)}^\infty \chi_3\mathcal U(t+kT,s)\psi_3e^{ik\theta}\right).\]
Applying (2.1), for $\im(\theta)>AT$, we find
\[e^{ik_2\theta}\sum_{k=-(k_2-1)}^\infty \chi_3\mathcal U(t+kT,s)\psi_3e^{ik\theta}=e^{ik_2\theta}\sum_{k=-(k_2-1)}^\infty \chi_3\mathcal U(t',0)\mathcal U((k_2-1)T+kT)\mathcal U(0,s-T)\psi_3e^{ik\theta}\]
and the finite speed of propagation implies
\[\begin{array}{l}\displaystyle e^{ik_2\theta}\sum_{k=-(k_2-1)}^\infty \chi_3\mathcal U(t+kT,s)\psi_3e^{ik\theta}\\
\ \\
\displaystyle=e^{ik_2\theta}\sum_{k=-(k_2-1)}^\infty \chi_3\mathcal U(t',0)\chi_4\mathcal U((k_2-1)T+kT)\psi_4\mathcal U(0,s-T)\psi_3e^{ik\theta}.\end{array}\]
Applying \eqref{lalali} to the right hand side term of the last formula, we obtain
\[e^{ik_2\theta}\sum_{k=-(k_2-1)}^\infty \chi_3\mathcal U(t+kT,s)\psi_3e^{ik\theta}=-\chi_3\mathcal U(t',0)R_{\chi_4,\psi_4}(\theta)\mathcal U(0,s-T)\psi_3.\]
It follows
\begin{equation}\label{tutu}F'(\chi_3V(t,s)\psi_3)(t,\theta)=P_1\left(e^{i\frac{t'}{T}\theta}\left[\chi_3\mathcal U(t',s)\psi_3-\chi_3\mathcal U(t',0)R_{\chi_4,\psi_4}(\theta)\mathcal U(0,s-T)\psi_3\right]\right)P^2.\end{equation}
Following the same argument, for $t'< s$ and $\im(\theta)>AT$, we get
\begin{equation}\label{tutu1}F'(\chi_3V(t,s)\psi_3)(t,\theta)=-P_1\left(e^{i\frac{t'}{T}\theta}\chi_3\mathcal U(t',0)R_{\chi_4,\psi_4}(\theta)\mathcal U(0,s-T)\psi_3\right)P^2.\end{equation}
Recall that $T_2(b)=k_0T$.
We have established in subsection 3.3 that, for $t\geq(k_0+1)T=T_2(b)+T$ and $0\leq s\leq\frac{2T}{3}<T$,  $V(t,s,\theta)$ admits a meromorphic continuation to
 $\{ \theta\in\mathbb C'\ :\ \textrm{Im}(\theta)\geq0\}$.
Moreover, from  \eqref{tutu} and \eqref{tutu1}, for $t\geq(k_0+1)T$ and $0\leq s\leq\frac{2T}{3}$, assumption $\rm(H2)$ implies that the family of operators  $V(t,s,\theta)$ has no poles on $\{ \theta\in\mathbb C'\ :\ \textrm{Im}(\theta)\geq0\}$ and satisfies \eqref{eq=thm2E}.
Thus,  the family of operators  $V(t,s,\theta)$ is analytic with respect to  $\theta$  on $\theta\in\{ \theta\in\mathbb C'\ :\ \textrm{Im}(\theta)\geq0\}$ and satisfies \eqref{eq=thm2E}.\end{proof}

Now, by integrating on a suitable contour of ${\C}'$, we obtain the following estimates.
\begin{lem}\label{l2}Assume $\rm(H1)$ and $\rm(H2)$  fulfilled and let $n\geq4$ be even. Then, for all $0\leq s\leq\frac{2T}{3}$ and for all $ d\in\mathbb{N}$   such that $d\geq k_0+1$, we have
\begin{equation}\label{eq=thm2P}\Vert \chi_3V(dT,s)\psi_3\Vert_{\mathcal L(L^2({\R}^n),\dot{H}^1({\R}^n))}\leq\frac{C_5}{(dT+1)\ln^2(dT+e)} .\end{equation}\end{lem}
\begin{proof}
In subsection 3.3, we have shown that $V(dT,s,\theta)=F'(\chi_3V(t,s)\psi_3)(dT,\theta)$ satisfies property (S'). Thus, $V(dT,s,\theta)$ admits a meromorphic continuation with respect to $\theta$ on $\mathbb C'$. Assumption  (H2) and Lemma 1 imply that $V(dT,s,\theta)$ has no poles on $\{ \theta\in\mathbb C'\ :\ \textrm{Im}(\theta)\geq0\}$. Moreover, $V(dT,s,\theta)$ is bounded independently of the choice of $s$ and satisfies \eqref{eq=thm2E}.
Also, there exists $\epsilon_0>0$ such that for $\theta\in\mathbb C'$ with $\vert\theta\vert\leq\epsilon_0$ we have
\begin{equation}\label{eq=thm2F}V(dT,s,\theta)=V((k_0+1)T,s,\theta)=\sum_{k\geq-m}\sum_{j\geq-m_k}R_{kj}\theta^k(\log\theta)^{-j}.\end{equation}
The property \eqref{eq=thm2E} implies that for the representation (\ref{eq=thm2F}) we have $R_{kj}=0$ for $k<0$ or $k=0$ and $j<0$. It follows that, for $\theta\in\mathbb C'$ with $\vert\theta\vert\leq\epsilon_0$, we obtain the following representation
\begin{equation}\label{eq=thm2G}V(dT,s,\theta)=V((k_0+1)T,s,\theta)=A(s,\theta)+B(s)\theta^{m_0}\log(\theta)^{-\mu}+ \underset{\theta\rightarrow0}o\left(\theta^{m_0}\log(\theta)^{-\mu}\right)\end{equation}
with $A(s,\theta)$ an holomorphic function with respect to $\theta$ for $\vert\theta\vert\leq\epsilon_0$ , $B(s)$ a finite-dimensional operator , $m_0\geq0$  and $\mu\geq 1$. Moreover, \eqref{tutu} and \eqref{tutu1} imply that $A(s,\theta)$ and $B(s)$ are bounded independently of $s$.

Since $V(dT,s,\theta)$ has no poles on $\{ \theta\in\mathbb C'\ :\ \textrm{Im}(\theta)\geq0\}$, there exists $\displaystyle0<\delta\leq\frac{\epsilon_0}{T}$ and $0<\nu<\epsilon_0$  sufficiently small such that  $V(dT,s,\theta)$ has no poles on  \[\{ \theta\in\mathbb C'\ :\ \textrm{Im}(\theta)\geq-\delta T ,\ -\pi\leq  \textrm{Re}(\theta)\leq-\nu,\ \nu\leq  \textrm{Re}(\theta)\leq\pi\}.\] Consider the contour $\gamma=\Gamma_1\cup\omega\cup\Gamma_2$ where $\Gamma_1=[-i\delta T-\pi,-i\delta T-\nu]$, $\Gamma_2=[-i\delta+\nu,-i\delta+\pi]$. The contour  $\omega$  of $\mathbb C$, is a curve connecting $-i\delta T-\nu$ and $-i\delta T+\nu$ symmetric with respect to the axis $ \textrm{Re}(\theta)=0$. The part of $\omega$ lying in $\{\theta\ :\  \textrm{Im}(\theta)\geq0\}$ is a half-circle
with radius $\nu$, $\omega\cap\{\theta\ :\   \textrm{Re}(\theta)<0,\  \textrm{Im}(\theta)\leq0\}=[-\nu-i\delta T,-\nu]$ and $\omega\cap\{\theta\ :\   \textrm{Re}(\theta)>0,\  \textrm{Im}(\theta)\leq0\}=[\nu,\nu-i\delta T]$. Thus, $\omega$ is included in the
region where we have no poles of $V(dT,s,\theta)$. Consider the closed contour 
\[\mathcal C=[i(A+1)T+\pi,i(A+1)T-\pi]\cup[i(A+1)T-\pi,-i\delta T-\pi]\cup\gamma\cup[-i\delta T+\pi,i(A+1)T+\pi].\]

The statement 2) of Proposition \ref{p8} implies
\begin{equation}\label{eq=thm2H}V(dT,s,\theta+2\pi)=V(dT,s,\theta).\end{equation}
Since the contour $\mathcal C$ is included in the
region where $V(dT,s,\theta)$ has no poles, the Cauchy formula implies
\[\int_{\mathcal C}e^{-id\theta}V(dT,s,\theta)\d\theta=0.\]
Moreover,  (\ref{eq=thm2H}) implies
\[\int_{[i(A+1)T-\pi,-i\delta T-\pi]}\hspace{-2cm}e^{-id\theta}V(dT,s,\theta)\d\theta=-\int_{[-i\delta T+\pi,i(A+1)T+\pi]}\hspace{-2cm}e^{-id\theta}V(dT,s,\theta)\d\theta\]
and we obtain
\begin{equation}\label{eq=thm2K}\displaystyle\int_{[i(A+1)T-\pi,i(A+1) T+\pi]}\hspace{-2cm}F(V(t,s))(dT,\theta)\d \theta=\int_\gamma F(V(t,s))(dT,\theta)\d \theta.\end{equation}
The formula (\ref{eq=VF}) and the identity (\ref{eq=thm2K}) imply
\begin{equation}\label{eq=thm2L}\chi_3 V(dT,s)\psi_3=\frac{1}{2\pi}\int_\gamma F(V(t,s))(dT,\theta)\psi_3\d\theta=\frac{1}{2\pi}\int_\gamma e^{-id\theta}V((k_0+1)T,s,\theta)\d\theta.\end{equation}
We will now estimate the right-hand side term of (\ref{eq=thm2L}).  Consider $A(s,\theta)$ the holomorphic part of the expansion (\ref{eq=thm2G}).
 Choose $\delta$ such that $\displaystyle\delta<\frac{\epsilon_0}{T}$. Then,  the closed contour $\omega\cup[-i\delta T-\nu,-i\delta T+\nu]$ is contained in the domain $\{\theta\in\mathbb C\ :\  \vert\theta\vert<\epsilon_0\}$. Since $A(s,\theta)$ is holomorphic with respect to $\theta$, for $\vert\theta\vert\leq\epsilon_0$, by applying the  Cauchy formula, we obtain
\[\int_{\omega} e^{-id\theta}A(s,\theta)\d\theta=-\int_{[-i\delta T-\nu,-i\delta T+\nu]}e^{-id\theta}A(s,\theta)\d\theta\]
and, since $A(s,\theta)$ is bounded independently of $s$, it follows
\begin{equation}\label{eq=thm2M}\left\vert\int_{\omega} e^{-id\theta}A(s,\theta)\d\theta\right\vert\leq  C_1e^{-\delta (dT)}\end{equation}
with $C_1>0$ independent of $s$ and $d$. For $\theta\in\Gamma_1\cup\Gamma_2$,  $V(dT,s,\theta)=V(k_0T,s,\theta)$ is bounded independently of $s$, and we show easily that
\begin{equation}\label{eq=thm2N}\left\vert\int_{\Gamma_j} e^{-ik\theta}V(dT,s,\theta)\d\theta\right\vert\leq C_2e^{-\delta (dT)},\quad j=1,2\end{equation}
with $C_2$ independent of $s$ and $d$.
Applying the estimates (\ref{eq=thm2M}), (\ref{eq=thm2N}) and the representation (\ref{eq=thm2G}), we get 
\begin{equation}\begin{array}{lll}\label{eq=thm2O}\int_\gamma e^{-id\theta}V((k_0+1)T,s,\theta)\d\theta&=& \displaystyle \underset{d\rightarrow+\infty}o\left(\frac{1}{(dT+1)\ln^2(dT+e)}\right)\\
\ \\
\ &\ &\displaystyle+ \int_\omega e^{-id\theta}\left(B(s)\theta^{m_0}(\log\theta)^{-\mu}+  \underset{\theta\rightarrow0}o\left(\theta^{m_0}(\log\theta)^{-\mu}\right)\right)\d\theta.\end{array}\end{equation}Following  Lemma 7 in Chapter IX of \cite{V1},  for $t=d$ and $\nu=\frac{1}{d}$,  we obtain
\[\int_\omega e^{-id\theta}\theta^{m_0}(\log\theta)^{-\mu} \d\theta \leq \frac{C_3}{(dT+1)^{m_0+1}\ln^{\mu+1}(dT+e)}.\]
Combining this estimate with the representation (\ref{eq=thm2O}), for all
 $d\geq k_0+1$ and  $s\in]0,T]$, we get
\[\left\Vert\int_\gamma e^{-id\theta}V((k_0+1)T,s,\theta)\d\theta\right\Vert_{\mathcal L(L^2,\dot{H}^1({\R}^n))}\leq\frac{C_4}{(dT+1)\ln^2(dT+e)}\]
with $C_4>0$ independent of $s$ and $d$. The inversion formula (\ref{eq=thm2L}) implies that, for all $d\geq k_0+1$ and  $0\leq s\leq\frac{2T}{3}$, we have \eqref{eq=thm2P}.
 
\end{proof}

\begin{lem}
Assume $\rm(H1)$ and $\rm(H2)$ fulfilled and let $n\geq4$ be even. Then, for all $0\leq s\leq\frac{2T}{3}$ and for  all $d\in\mathbb{N}$ such that $d\geq k_0+1$, we have

\begin{equation}\label{eq=thm2W}\left\|\chi_3 \partial_t V(dT,s) \psi_3\right\|_{{\mathcal L}(L^2(\R^n), L^2(\R^n))} \leq \frac{C_5}{(dT + 1)\ln^2(dT + e)}.\end{equation}
\end{lem}
\begin{proof}

For $\im(\theta)>AT$, we have 
\[\partial_tF'(\chi_3V(t,s)\psi_3)(t,\theta)=\frac{i\theta}{T}F'(\chi_3V(t,s)\psi_3)(t,\theta)+F'(\chi_3\partial_tV(t,s)\psi_3)(t,\theta)\]
and it follows that 
\[F'(\chi_3\partial_tV(t,s)\psi_3)(t,\theta)=\partial_tF'(\chi_3V(t,s)\psi_3)(t,\theta)-\frac{i\theta}{T}F'(\chi_3V(t,s)\psi_3)(t,\theta).\]
Since, for $t\geq (k_0+1)T$ and $0\leq s\leq \frac{2T}{3}$,  the family of operators $F'(\chi_3V(t,s)\psi_3)(t,\theta)$ satisfies property (S'), $\partial_tF'(\chi_3V(t,s)\psi_3)(t,\theta)$ satisfies also (S'). Thus, the family of operators $F'(\chi_3\partial_tV(t,s)\psi_3)(t,\theta)$ admits a meromorphic continuation satisfying property (S'). Moreover, following the definition of $\mathcal U(t,s)$, we have
\[\partial_tV(t,s)=P_2\mathcal U(t,s)P^2\]
and we get
\[\chi_3 \partial_t V(t,s) \psi_3=P_2\chi_3 \mathcal U(t,s) \psi_3P^2.\]
Following the same arguments as those used in the proof of Lemma 1 , we obtain
\[F'(\chi_3\partial_tV(t,s)\psi_3)(t,\theta)=P_2\left(e^{i\frac{t'}{T}\theta}\left[\chi_3\mathcal U(t',s)\psi_3-\chi_3\mathcal U(t',0)R_{\chi_4,\psi_4}(\theta)\mathcal U(0,s-T)\psi_3\right]\right)P^2\]
with $t=lT+t'$, $l\in\mathbb N$ and $0\leq t'<T$. Thus, assumption (H2) implies that, for for $t\geq (k_0+1)T$ and $0\leq s\leq \frac{2T}{3}$, $F'(\chi_3\partial_tV(t,s)\psi_3)(t,\theta)$ is analytic with respect to $\theta$ on $\{ \theta\in{\C}'\ :\ \im(\theta)\geq0\}$ and 
\[\displaystyle\limsup_{\substack{\lambda\to0 \\ \im(\lambda)>0}}\left(\sup_{s\in[0,\frac{2T}{3}]}\Vert F'(\chi_3\partial_tV(t,s)\psi_3)(t,\lambda)\Vert_{\mathcal L(L^2({\R}^n),\dot{H}^1({\R}^n))}\right)<\infty.\] 
Following the same arguments as those used in the proof of Lemma 2, we obtain \eqref{eq=thm2W}.\end{proof}

\ \\
\textit{Proof of Theorem \ref{t3}.} Let $\alpha\in\mathcal C^\infty({\R})$ be such that $\alpha(t)=0$ for $t\leq\frac{T}{2}$ and $\alpha(t)=1$ for $t\geq \frac{2T}{3}$. For all $h\in\dot{H}^1({\R}^n)$, $w_1=\alpha(t)U(t,0)h$ is the solution of
\begin{equation}\label{eq=thm2R} \left\{\begin{array}{c}
\partial_t^2w_1-\Div_{x}(a(t,x)\nabla_{x}w_1)=[\partial_t^2,\alpha](t)U(t,0)h,\\
(w_1,\partial_tw_1)_{\vert t=0}=(0,0).\end{array}\right.\end{equation}
We deduce from the Cauchy problem (\ref{eq=thm2R}) the following representation
\begin{equation}\label{eq=thm2S}U(t,0)=\alpha(t)U(t,0)=\int_0^tV(t,s)[\partial_t^2,\alpha](s)U(s,0)\d s,\quad t\geq T.\end{equation}
Since $[\partial_t^2,\alpha](t)=0$ for $t>\frac{2T}{3}$, the formula (\ref{eq=thm2S}) becomes
\[U(t,0)=\int_0^{\frac{2T}{3}}V(t,s)[\partial_t^2,\alpha](s)U(s,0)\d s,\quad  t\geq T.\]
The finite speed of propagation implies
\begin{equation}\label{eq=thm2T}\chi_2U(dT,0)\psi_2=\int_0^{\frac{2T}{3}}\chi_2V(dT,s)\psi_3[\partial_t^2,\alpha](s)U(s,0)\psi_2\d s,\quad  d\geq1.\end{equation}
The formula (\ref{eq=thm2T}) and the estimate (\ref{eq=thm2P}) imply that, for
 $d\geq k_0+1$, we have
\begin{equation}\label{eq=thm2U}\Vert \chi_2U(dT,0)\psi_2\Vert_{\mathcal L(\dot{H}^1({\R}^n),\dot{H}^1({\R}^n))}\leq\frac{C_6}{(dT+1)\ln^2(dT+e)},\end{equation}
with $C_6>0$ independent of $d$. 
Let $\beta\in{\CI}$. The formula (\ref{eq=thm2S}) implies that, for $t\geq (k_0+1)T$, we have
\[\partial _tU(t,0)\beta=\int_0^{\frac{2T}{3}}\partial_tV(t,s)[\partial_t^2,\alpha](s)U(s,0)\beta \d s.\]
By density, this leads to
\[\label{eq=thm2X}\chi_2\partial _tU(dT,0)\psi_2=\int_0^{\frac{2T}{3}}\chi_2\partial _tV(dT,s)\psi_3[\partial_t^2,\alpha](s)U(s,0)\psi_2\d s,\quad  d\geq k_0+1\]
and the estimate (\ref{eq=thm2W}) implies that for $d\geq k_0+1$ we get
\begin{equation}\label{eq=thm2Y}\Vert \chi_2\partial_tU(dT,0)\psi_2\Vert_{\mathcal L(\dot{H}^1({\R}^n),L^2({\R}^n))}\leq\frac{C_9}{(dT+1)\ln^2(dT+e)}.\end{equation}
The estimates (\ref{eq=thm2P}), (\ref{eq=thm2W}), (\ref{eq=thm2U}) and (\ref{eq=thm2Y}), imply that, for $d\geq k_0+1$, we have
\begin{equation}\label{eq=thm2Z}\Vert \chi_2\mathcal U(dT,0)\psi_2\Vert_{\mathcal L({\B}({\R}^n))}\leq\frac{C_9}{(dT+1)\ln^2(dT+e)}.\end{equation}
 Assume $t-s\geq (k_0+3)T$ and choose $k,l\in\mathbb{N}$ such that
\[kT\leq t\leq(k+1)T,\quad lT\leq s\leq (l+1)T.\]
Then, statement (\ref{eq=prop1A}) and the finite speed of propagation imply
\[\chi_1\mathcal U(t,s)\psi_1=\chi_1\mathcal U(t,kT)\chi_2\mathcal U((k-(l+1))T)\psi_2\mathcal U((l+1)T,s)\psi_1\]
and $(k-(l+1))T\geq (k_0+1)T$. Combining estimates (\ref{eq=prop2A}), (\ref{eq=thm2Z}), we get 
\[\Vert\chi_1\mathcal U(t,s)\psi_1\Vert_{\mathcal L({\B}}\leq\frac{C_{10}}{((k-(l+1))T+1)\ln^2((k-(l+1))T+e)}.\]
Moreover, we find
\[\begin{array}{lll}(t-s+1)\ln^2(t-s+e)&\leq&((k-(l+1))T+2T+1)\ln^2((k-(l+1))T+2T+e)\\
\ &\leq&\displaystyle(k-(l+1))T\ln^2((k-(l+1))T)\left(1+\frac{2T+1}{(k-(l+1))T}\right)\\
\ \\
\ &\ &\displaystyle\times\left(1+\frac{\ln(1+\frac{2T+e}{(k-(l+1))T})}{\ln((k-(l+1))T)}\right)^2\\
\ \\
\ &\leq& C_{11}(k-(l+1))T\ln^2((k-(l+1))T)\end{array}\]
and we show easily that
\[(k-(l+1))T\ln^2((k-(l+1))T)\leq C_{12}((k-(l+1))T+1)\ln^2((k-(l+1))T+e).\]
We deduce the estimate
\[(t-s+1)\ln^2(t-s+e)\leq C_{13}((k-(l+1))T+1)\ln^2((k-(l+1))T+e).\]
Finally, it follows that
\[\Vert\chi_1\mathcal U(t,s)\psi_1\Vert_{\mathcal L({\B}({\R}^n))}\leq\frac{C_{14}}{(t-s+1)\ln^2(t-s+e)}.\]
For $t-s\leq (k_0+3)T$, following estimate (\ref{eq=prop2A}), we have
\[\Vert\chi_1\mathcal U(t,s)\psi_1\Vert_{\mathcal L({\B}({\R}^n))}\leq C_{15}e^{A(k_0+3)T}\leq C_{15}e^{A(k_0+3)T}\left(\frac{((k_0+3)T+1)\ln^2((k_0+3)T+e)}{(t-s+1)\ln^2(t-s+e)}\right).\]
Then, we obtain (\ref{eq=thm2A}) for $n\geq4$ even.$\square$

\section{$L^2$ integrability of the local energy}
The purpose of this section is to show the $L^2$ integrability of the local energy by applying estimate (\ref{eq=thm2A}). For the free wave equation Smith and Sogge have established the following result
\begin{lem}\label{l2}\emph{(\cite{SS}, Lemma 2.2)}
Let $\gamma\leq \frac{n-1}{2}$ and let  $\phi\in\mathcal C^\infty_0(\vert x\vert<\rho+1)$. Then
\begin{equation}\label{eq=L20}\int_{\R}\Vert \phi e^{\pm it\Lambda}h\Vert_{H^\gamma({\R}^n)}^2\d t\leq C(\phi,n,\gamma)\Vert h\Vert_{\dot{H}^\gamma({\R}^n)}^2,\quad h\in\dot{H}^\gamma({\R}^n).\end{equation}\end{lem}
In \cite{SS} the authors consider only odd dimensions $n\geq3$, but the proof of this lemma goes without any change for even dimensions.
We deduce from (\ref{eq=L20}) the following estimate.
\begin{lem}\label{l3}
 Let $\gamma\leq\frac{n-1}{2}$ and $\phi\in{\CI}$. Then
\begin{equation}\label{eq=L200}\int_{\R}\Vert \phi U_0(t)f\Vert^2_{\dot{\mathcal{H}}_\gamma({\R}^n)}\d t\leq C(\phi,n,\gamma)\Vert f\Vert_{\dot{\mathcal{H}}_\gamma({\R}^n)}^2,\quad f\in\dot{\mathcal{H}}_\gamma({\R}^n).\end{equation}
\end{lem}

Following estimates  (\ref{eq=thm2A}) and (\ref{eq=L200}), we will establish the $L^2$ integrability of the local energy which take the following form:
\begin{Thm}\label{t12}
 Assume $n\geq4$  even and  $\rm(H1)$, $\rm(H2)$ fulfilled. Then, for all $\phi\in\mathcal C_0^\infty(\vert x\vert\leq\rho+1)$, we have
\begin{equation}\label{eq=thm8A}\int_0^\infty\Vert \phi \mathcal{U}(t,0)f\Vert_{{\B}}^2dt\leq C(T,\phi,n,\rho)\Vert f\Vert_{{\B}}^2.\end{equation}
\end{Thm}
\begin{proof}
 Choose  $f\in {\B}$ and $\chi\in\mathcal{C}^\infty_0(\vert x\vert<\rho+1)$ such that
 $\chi=1$ for $ \vert x\vert\leq \rho +\half$ and $0\leq\chi\leq1$. Notice that
\begin{equation}\label{eq=thm8B}\phi \mathcal{U}(t,0)f=\phi \mathcal{U}(t,0)\chi f+ \phi \mathcal{U}(t,0)(1-\chi)f.\end{equation}
Then, combining estimates (\ref{eq=thm2A}) and \eqref{eq=L20}, we deduce \eqref{eq=L200} (see the proof of Theorem 4 in \cite{K1}).\end{proof}
\ \\
\textit{Proof of Theorem \ref{t1}.} Applying the equivalence of assumptions (H2) and (H3) for $n\geq3$ odd, we obtain (\ref{eq=estC}) for $n\geq3$ odd (see Remark 1).
Then, combining estimates (\ref{eq=thm8A}),  (\ref{eq=thm2A})    and the local estimates (\ref{eq=estB}), we deduce (\ref{eq=estC}) for $n\geq4$ even (see  \cite{K1} for more details). \hspace{12cm}$\square$

\section{ Examples of metrics $a(t,x)$} 

In this section we will apply the results for non-trapping metrics  independent of $t$  to construct time periodic metrics such that conditions (H1) and (H2) are fulfilled.  Consider the following condition
\begin{equation}\label{eq=exA}
\frac{2a}{\rho} -\frac{\vert a_t\vert}{\sqrt{\inf a}}-\vert a_r\vert\geq\beta>0
\end{equation}
with $\beta$ independent of $t$ and $x$.  It has been established that assumption  (H1) is fulfilled if $a(t,x)$ satisfies (\ref{eq=exA}) (see \cite{K1}). Thus, we suppose that $a(t,x)$ satisfies (\ref{eq=exA}) and we will introduce conditions that imply  (H2). In \cite{MT} and \cite{MT1}, Metcalfe and Tataru have established local energy decay for the solution of wave equation with time dependent perturbations, by assuming that the perturbations of the D'Alambertian ($a(t,x)-1$ for the problem \eqref{eq=lepbA}) is sufficiently small. Set
\[D_0=\{ x\ :\ \abs{x}\leq2\},\quad D_j=\{ x\ :\ 2^j\leq\abs{x}\leq2^{j+1}\},\quad j=1,2,\cdots\]
and
\[A_j={\R}\times D_j.\]
For \eqref{eq=lepbA}, the main assumption of \cite{MT} and \cite{MT1} takes the form
\[\sum_{j=0}^\infty\left(\sup_{(t,x)\in A_j}\left[\ \left\langle x\right\rangle^2\norm{\partial^2_xa(t,x)}+\left\langle x\right\rangle\abs{\nabla_xa(t,x)}+\abs{a(t,x)-1}\ \right]\right)\leq\epsilon\]
with $\epsilon>0$ sufficiently small. For $\epsilon$ sufficiently small, this condition implies that \eqref{eq=lepbA} is non-trapping (see \cite{MT1}). Thus, Metcalfe and Tataru have shown local energy decay by modifying the size of one parameter of the metric. Following this idea, we will establish examples of metrics such that (H2) is fulfilled by modifying the size $T$ of the period of $a(t,x)$. This choice is justified by the properties of $\mathcal U(t,s)$.

Let $T_1>0$ and let $(a_T)_{T\geq T_1}$ be a family of functions such that  $a_T(t,x)$  is $T$-periodic with respect to  $t$  and $a_T(t,x)$ satisfies (\ref{eq=perturbationA}) and (\ref{eq=exA}). Moreover, assume that 
\begin{equation}\label{eq=exB} a_T(t,x)=a_1(x),\quad  t\in[T_1,T],\   x\in{\R}^n.\end{equation} 
Notice that (\ref{eq=exA}) implies that $a_1(x)$ is non-trapping (see \cite{K1}). We will show that for $T$ sufficiently large  (H2) will be fulfilled for $a(t,x)=a_T(t,x)$. Notice that for $n\geq3$ odd , it has been proved in \cite{K1} that, for $T$ large enough, (\ref{eq=exA}) and (\ref{eq=exB}) imply (H3). Combing this result with Theorem \ref{t2}, we find that, for $n\geq3$ odd and  for $T$ large enough, (\ref{eq=exA}) and (\ref{eq=exB})  imply (H2). It  remains only to treat the case $n\geq4$ even. 

Consider the following Cauchy problem
\begin{equation} \label{eq=exC}  \left\{\begin{array}{c}
v_{tt}-\Div_{x}(a_1(x)\nabla_{x}v)=0,\\
(v,v_{t})(0)=f,\end{array}\right.\end{equation}
and the associated propagator
\[\mathcal{V}(t):{\B}\ni f\longmapsto (v,v_t)(t)\in{\B}.\]
Let $u$ be solution of (\ref{eq=lepbA}). For $T_1\leq t\leq T$ we have
\[\partial_t^2 u -\Div_x(a_1(x)\nabla_x u)=\partial_t^2 u -\Div_x(a_T(t,x)\nabla_x u)=0.\]
It follows that for $a(t,x)=a_T(t,x)$ we get
\begin{equation} \label{eq=exD}\mathcal{U}(t,s)=\mathcal{V}(t-s),\quad T_1\leq s<t\leq T.\end{equation}
The asymptotic behavior, when $t\rightarrow+\infty$, of the local energy of (\ref{eq=lepbA}), assuming $a(t,x)$ is independent of $t$, has been studied by many authors (see  \cite{V1}, \cite{V3} and \cite{V4}). It has been proved that, for non-trapping metrics and for $n\geq3$, the local energy decreases. To prove (H2), we will apply the following result. 
\begin{Thm}\label{t13}
Assume $n\geq4$  even. Let $\phi_1,\phi_2\in{\CI}$. Then, we have
\begin{equation} \label{eq=thm11A}\Vert \phi_1\mathcal{V}(t)\phi_2\Vert_{\mathcal{L}({\B})}\leq C_{\phi_1,\phi_2}\left\langle t\right\rangle^{1-n},\end{equation}
with $C_{\phi_1,\phi_2}>0$ independent of $t$.\end{Thm}
Estimate (\ref{eq=thm11A}) has been established by Vainberg in \cite{V1}, \cite{V3} but also by Vodev in \cite{V5} and \cite{V4}.
  For $n\geq4$ even we will use the following identity.

\begin{lem}\label{l4}
Let $\psi\in\mathcal{C}^\infty_0(\vert x\vert\leq\rho+1+T_1)$ be such that $\psi=1$, for $\vert x\vert\leq\rho+\half+T_1$. Then, we have
\begin{equation} \label{eq=lem5A}\mathcal{U}(T_1,0)-\mathcal{V}(T_1)=\psi(\mathcal{U}(T_1,0)-\mathcal{V}(T_1))=(\mathcal{U}(T_1,0)-\mathcal{V}(T_1))\psi.\end{equation}
\end{lem}
\begin{proof}
Choose $g\in{\B}$ and let $w$ be the function defined by $(w,w_t)(t)=\mathcal{U}(t,0)(1-\psi)g$. The finite speed of propagation implies that, for  $0\leq t\leq T_1$ and $\vert x\vert\leq\rho+\half$, we have $w(t,x)=0$. Then, we obtain
\begin{equation} \label{eq=lem5B} \Div_x(a_1(x)\nabla_x)=\Delta_x=\Div_x(a(t,x)\nabla_x),\quad\textrm{for } \vert x\vert>\rho.\end{equation}
Thus, $w$ is solution on $0\leq t\leq T_1$ of the problem
\[  \left\{\begin{array}{c}
w_{tt}-\Div_{x}(a_1(x)\nabla_{x}w)=0,\\
(w,w_{t})(0)=(1-\psi)g\end{array}\right.\]
and it follows that
\begin{equation} \label{eq=lem5C}(\mathcal{U}(T_1,0)-\mathcal{V}(T_1))(1-\psi)=0.\end{equation}
Now, let  $u$ and $v$ be the functions defined by $(u,u_t)(t)=\mathcal{U}(t,0)g$ and $(v,v_t)(t)=\mathcal{V}(t)g$ with $g\in{\B}$. Applying (\ref{eq=lem5B}), we can easily show that on  $(1-\psi)u$ is the solution of
\[  \left\{\begin{array}{c}
\partial_t^2((1-\psi)u))-\Delta_x((1-\psi)u))=[\Delta_x,\psi]u,\\
(((1-\psi)u),((1-\psi)u)_{t})(0)=(1-\psi)g,\end{array}\right.\]
and $(1-\psi)v$ is the solution of
\[  \left\{\begin{array}{c}
\partial_t^2(((1-\psi)v))-\Delta_x((1-\psi)v))=[\Delta_x,\psi]v,\\
(((1-\psi)v),((1-\psi)v)_{t})(0)=(1-\psi)g.\end{array}\right.\]
We have
\begin{equation} \label{eq=lem5D}(1-\psi)(\mathcal{U}(T_1,0)-\mathcal{V}(T_1))=0.\end{equation}
Combining (\ref{eq=lem5C}) and (\ref{eq=lem5D}), we get (\ref{eq=lem5A}).\end{proof}

From now on, we consider the cut-off function $\psi\in\mathcal{C}^\infty_0(\vert x\vert\leq\rho+1+T_1)$ such that $\psi=1$, for $\vert x\vert\leq\rho+\half+T_1$.

\begin{lem}\label{l5} Assume $n\geq4$  even and let $(a_T)_{T\geq T_1}$ satisfy \eqref{eq=exA} and \eqref{eq=exB}. Then, for $T$ large enough and for $a(t,x)=a_T(t,x)$, we have
\begin{equation} \label{eq=lem6A}
\mathcal{U}(NT,0)\psi=\mathcal{V}(NT)\psi+\sum_{k=0}^{N-1}\mathcal{V}(kT+T-T_1)B_N^k,\quad  N\geq1,\end{equation}
where, for all $N\geq1$ and all $k\in\{0,\ldots,N-1\}$, $B_N^k$ satisfies
\begin{equation} \label{eq=lem6B}  \left\{\begin{array}{c}
B_N^k=\psi B_N^k,\\
\ \\
\displaystyle\Vert B_N^k\Vert_{\mathcal{L}({\B})}\leq\frac{C}{(N-k)\ln^2(N-k+e)}
\end{array}\right.\end{equation}
with $C>0$ independent of $N$, $k$ and $T$.
\end{lem}
\begin{proof}
We will show (\ref{eq=lem6A}) and (\ref{eq=lem6B}), by induction. First, set
\[B^0_1=\mathcal{U}(T_1,0)-\mathcal{V}(T_1).\]
We deduce from (\ref{eq=lem5A}) that
\begin{equation} \label{eq=lem6B1}B^0_1=\psi B^0_1=B^0_1\psi.\end{equation}
Moreover, statement (\ref{eq=exD}) implies
\begin{equation} \label{eq=lem6C}\mathcal{U}(T,0)=\mathcal{V}(T-T_1)\mathcal{U}(T_1,0)=\mathcal{V}(T-T_1)B^0_1+\mathcal{V}(T).\end{equation}
Combining (\ref{eq=lem6B1}) and (\ref{eq=lem6C}), we can see that (\ref{eq=lem6A}) is true for $N=1$. 
Now, assume (\ref{eq=lem6A}) and (\ref{eq=lem6B}) hold for $N\geq1$. Set
 $S=\mathcal{U}(T_1,0)-\mathcal{V}(T_1)$. Using(\ref{eq=lem5A}) we get
\begin{equation} \label{eq=lem6D} S=\psi S=S\psi.\end{equation}
Then,  we obtain
 \[\mathcal{U}((N+1)T,0)\psi=\mathcal{U}(T,0)\mathcal{U}(NT,0)\psi=(\mathcal{V}(T)+\mathcal{V}(T-T_1)S)\mathcal{U}(NT,0)\psi.\]
The induction assumption yields
\[\mathcal{U}((N+1)T,0)\psi=(\mathcal{V}(T)+\mathcal{V}(T-T_1)S)\left(\mathcal{V}(NT)\psi+\sum_{k=0}^{N-1}\mathcal{V}(kT+T-T_1)B_N^k\right),\]
where, for all $k\in\{0,\ldots,N-1\}$, $B_N^k$ satisfies (\ref{eq=lem6B}).
It follows that
\begin{equation} \label{eq=lem6E}\mathcal{U}((N+1)T,0)\psi=\mathcal{V}((N+1)T)\psi+\sum_{k=0}^{N}\mathcal{V}(kT+T-T_1)B_{N+1}^k,\end{equation}
where, for all $k\in\{1,\ldots,N\}$, $B^k_{N+1}=B^{k-1}_N$ and 
\[B^0_{N+1}=\sum_{k=0}^{N-1} S\mathcal{V}(kT+T-T_1)B_N^k+S\mathcal{V}(NT)\psi.\]
The induction assumption implies that, for all $k\in\{1,\ldots,N\}$, $B^k_{N+1}=B^{k-1}_N$ satisfies (\ref{eq=lem6B}). To conclude, it only remain to show that $B^0_{N+1}$ satisfies (\ref{eq=lem6B}). First,  (\ref{eq=lem6E}) implies
\begin{equation} \label{eq=lem6F}B^0_{N+1}=\psi B^0_{N+1}\end{equation}
and we get
\begin{equation} \label{eq=lem6G}B^0_{N+1}=\sum_{k=0}^{N-1} S\psi \mathcal{V}(kT+T-T_1)\psi B_N^k+S\psi\mathcal{V}(NT)\psi.\end{equation}
Estimate (\ref{eq=lem6B}) implies that, for $k\in\{0,\ldots,N-1\}$, we have
\begin{equation} \label{eq=lem6H}\Vert B_N^k\Vert_{\mathcal{L}({\B})}\leq\frac{C}{(N-k)\ln^2(N-k+e)},\end{equation} 
with $C>0$ independent of $k$, $N$ and $T$.
From (\ref{eq=thm11A}), for all $k\in\{0,\ldots,N\}$, we obtain
\[\Vert\psi\mathcal{V}(kT+T-T_1)\psi\Vert_{\mathcal{L}({\B})}\leq\frac{C_\psi}{(kT+1+T-T_1)\ln^2(kT+(T-T_1)+e)}.\]
If we choose $T\geq2$, the last inequality becomes
\begin{equation} \label{eq=lem6I}\Vert\psi \mathcal{V}(kT+T-T_1)\psi\Vert_{\mathcal{L}({\B})}\leq\frac{C_1}{T(k+1)\ln^2(k+1+e)},\end{equation}
where $C_1=2C(T_1)$ is independent of $k$, $N$ and $T$. Notice that $\Vert S\Vert$ is independent of $T$, $k$ and $N$. Combining representation  (\ref{eq=lem6G}) and estimates (\ref{eq=lem6H}), (\ref{eq=lem6I}), we find
\begin{equation} \label{eq=lem6J}\begin{array}{lll}\Vert B^0_{N+1}\Vert_{\mathcal{L}({\B})}&\leq& \displaystyle\frac{C_1C}{T}\sum_{k=0}^{N-1}\frac{1}{(N-k)\ln^2(N-k+e)}\cdot\frac{1}{(k+1)\ln^2(k+1+e)}\\
\ \\
\ &\ &\displaystyle+\frac{C_1}{(NT+1)\ln^2(N+1+e)}.\end{array}\end{equation}
Thus, we get
\begin{equation} \label{eq=lem6M}\Vert B^0_{N+1}\Vert_{\mathcal{L}({\B})}\leq\frac{4CC_1C_2+2C_1}{T}\cdot\frac{1}{(N+1)\ln^2(N+1+e)}.\end{equation}
 It follows from estimate (\ref{eq=lem6M}) and statement (\ref{eq=lem6F}) that if we choose $T$ such that $T\geq2$ and $\displaystyle\frac{4CC_1C_2+2C_1}{T}\leq C$, $B^0_{N+1}$ will satisfy (\ref{eq=lem6B}). Since the value of  $T$ is independent of $N$, by combining this result with (\ref{eq=lem6B1}) and (\ref{eq=lem6C}), we deduce that (\ref{eq=lem6A}) and (\ref{eq=lem6B}) hold for all  $N\geq1$.\end{proof}
From now on, we set $\beta\in\mathcal{C}^\infty_0(\vert x\vert\leq\rho+\frac{1}{4})$ such that $\beta=1$ for $\vert x\vert\leq\rho+\frac{1}{5}$.
\begin{lem}\label{l6}
Assume  $n\geq4$  even and let $(a_T)_{T\geq T_1}$ satisfy \eqref{eq=exA} and \eqref{eq=exB}. Let $s\in[T_1,T]$. Then, for $T$ large enough and $a(t,x)=a_T(t,x)$, we obtain
\begin{equation} \label{eq=lem7A}
\mathcal{U}(NT,s)\beta=\mathcal{V}(NT-s)\beta+\sum_{k=0}^{N-1}\mathcal{V}(kT+T-T_1)D_N^k(s),\quad  N\geq2,\end{equation}
where, for all $N\geq2$ and all $k\in\{0,\ldots,N-1\}$, $D_N^k(s)$ satisfies
\begin{equation} \label{eq=lem7B}  \left\{\begin{array}{c}
D_N^k(s)=\psi D_N^k(s),\\
\ \\
\displaystyle\Vert D_N^k(s)\Vert_{\mathcal{L}({\B})}\leq\frac{C}{(N-k)\ln^2(N-k+e)}
\end{array}\right.\end{equation}
with $C>0$ independent of $s$, $N$, $k$ and $T$.
\end{lem}
\begin{proof}
Since $s\in[T_1,T]$, we have $\mathcal U(T,s)=\mathcal V(T-s)$. It follows that
\[ \mathcal U(2T,s)\beta=(\mathcal V(T)+\mathcal V(T-T_1)S)\mathcal V(T-s)\beta=\mathcal V(2T-s)\beta+\mathcal V(T-T_1)D^1_2(s),\]
where $D^1_2(s)=S\mathcal V(T-s)\beta$.  Taking into account estimate (\ref{eq=thm11A}), it is easy to see that $D^1_2(s)$ satisfies (\ref{eq=lem7B}) for $N=2$. Consequently, repeating the arguments used for proving (\ref{eq=lem6A}) and (\ref{eq=lem6B}), we deduce that for $T$ large enough (\ref{eq=lem7A}) and (\ref{eq=lem7B}) are satisfied for all integers $N\geq2$.\end{proof}
\begin{lem}\label{l7}
Assume $n\geq4$  even and let $(a_T)_{T\geq T_1}$ satisfy \eqref{eq=exA} and \eqref{eq=exB}. Assume also that conditions \eqref{eq=lem6A}, \eqref{eq=lem6B}, \eqref{eq=lem7A} and \eqref{eq=lem7B}, are fulfilled for $T>2$ and $a(t,x)=a_T(t,x)$. Then, for all  $N\geq1$,\\
 $\phi_1\in\mathcal{C}^\infty_0(\vert x\vert\leq\rho+1+3T)$  and all $0\leq s\leq NT$, we have
\begin{equation} \label{eq=lem8A}\displaystyle \Vert\phi_1\mathcal U(NT,0)\psi\Vert\leq\frac{C}{(N+1)\ln^2(N+e)},\end{equation}
\begin{equation} \label{eq=lem8B}\displaystyle \Vert\phi_1\mathcal U(NT,s)\beta\Vert\leq\frac{C'}{(NT-s+1)\ln^2(NT-s+e)}\end{equation}
with $C,C'>0$ independent of $s$ and $N$.
\end{lem}
\begin{proof}
Since $T>2$,  estimate (\ref{eq=thm11A}) implies
\begin{equation} \label{eq=lem8C}\Vert \phi_1\mathcal{V}(kT)\phi_2\Vert_{\mathcal{L}({\B})}\leq \frac{C_2}{(k+1)\ln^2(k+e)},  k\in\mathbb N,\end{equation}
with $C_2$ independent of $k$. The representation (\ref{eq=lem6A}) can be written in the form
\[\phi_1\mathcal{U}(NT,0)\psi=\phi_1\mathcal{V}(NT)\psi+\sum_{k=0}^{N-1}\phi_1\mathcal{V}(kT+T-T_1)\psi B_N^k.\]
Combining this representation with estimates  (\ref{eq=lem6B}) and (\ref{eq=lem8C}), we get
\[\begin{array}{lll}\displaystyle\Vert \phi_1\mathcal{U}(NT,0)\beta\Vert_{\mathcal{L}({\B})}&\displaystyle\leq&\displaystyle\frac{C_3}{(N+1)\ln^2(N+e)}\\
\ \\
\displaystyle\ &\ &\displaystyle+C'_3\sum_{k=0}^{N-1}\frac{1}{(N-k)\ln^2(N-k+e)}\cdot\frac{1}{(k+1)\ln^2(k+1+e)}\end{array}\]
and this estimate implies (\ref{eq=lem8A}). Let $s\in[0,NT]$ and let $l\in\{0,\ldots,N\}$ be such that $s=lT+s'$, with $0\leq s'<T$. We have $\mathcal U(NT,s)=\mathcal U((N-l)T,s')$. We start by assuming $s'\in[T_1,T]$. Applying (\ref{eq=lem7A}) and (\ref{eq=lem7B}), for $N-l\geq2$ we obtain
\begin{equation} \label{eq=lem8D}\phi_1\mathcal{U}(NT,s)\beta=\phi_1\mathcal{U}((N-l)T,s')\beta=\phi_1\mathcal{V}((NT-s)\beta+\sum_{k=0}^{N-l-1}\phi_1\mathcal{V}(kT+T-T_1)\psi D_N^k(s'),\end{equation}
where $D_N^k(s')$ satisfying (\ref{eq=lem7B}). Combining estimates (\ref{eq=lem7B}), (\ref{eq=thm11A}) and the representation (\ref{eq=lem8D}), we obtain
\[\begin{array}{lll}\displaystyle\Vert \phi_1\mathcal{U}(NT,s)\beta\Vert_{\mathcal{L}({\B})}&\leq&\displaystyle\frac{C_4}{(N-l+1)\ln^2(N-l+e)}\\
\ \\
\displaystyle\ &\ &\displaystyle+C_4'\sum_{k=0}^{N-l-1}\frac{1}{(N-l-k)\ln^2(N-l-k+e)}\cdot\frac{1}{(k+1)\ln^2(k+1+e)},\end{array}\]
with $C_4,C_4'>0$  independent of $l$, $s'$ and $N$. Thus, we get
\begin{equation} \label{eq=lem8F}\displaystyle \Vert\phi_1\mathcal U(NT,s)\beta\Vert\leq\frac{C_5}{(N-l+1)\ln^2(N-l+e)}.\end{equation}
Notice that
\[\frac{(N-l+1)\ln^2(N-l+e)}{(NT-s+T)\ln^2(NT-s+Te)}\leq C_6\]
with $C_6$ independent of $s$, $N$ and $l$. Consequently, condition  (\ref{eq=lem8F}) implies (\ref{eq=lem8B}). For $N-l=1$, we have  $\mathcal U(NT,s)=\mathcal V(NT-s)$ and we deduce easily (\ref{eq=lem8B}). 

Now, assume  $s'\in[0,T_1]$. The finite speed of propagation implies
\[\phi_1\mathcal{U}(NT,s)\beta=\phi_1\mathcal{U}((N-l)T,s')\beta=\phi_1\mathcal{U}((N-l)T,0)\psi\mathcal{U}(0,s')\beta\]
and we obtain  (\ref{eq=lem8B}) by applying (\ref{eq=lem8A}).\end{proof}

\begin{Thm}\label{t14}
Assume $n\geq4$  even and let $(a_T)_{T\geq T_1}$ satisfy \eqref{eq=exA} and \eqref{eq=exB}. Then, for $T$ large enough and for $a(t,x)=a_T(t,x)$, assumption $\rm(H2)$ is fulfilled.
\end{Thm}
\begin{proof}
Choose $T\geq2$ such that conditions (\ref{eq=lem6A}), (\ref{eq=lem6B}), (\ref{eq=lem7A}) and (\ref{eq=lem7B}) are fulfilled, and set $\phi_1,\phi_2\in\mathcal C^\infty_0(\vert x\vert\leq\rho+2+3T)$ satisfying $\phi_i=1$ for $\vert x\vert\leq\rho+3T+1$, $i=1,2$. Let $\chi\in\mathcal{C}^\infty_0(\vert x\vert\leq\rho+\frac{1}{4})$ be such that $\chi=1$, for $\vert x\vert\leq\rho+\frac{1}{5}$. Consider the following representation
\begin{equation} \label{eq=thm12A}\phi_1\mathcal U(NT,0)\phi_2=\phi_1\mathcal U(NT,0)\chi\phi_2+\phi_1\mathcal U(NT,0)(1-\chi)\phi_2.\end{equation}
For the first term on the right hand side of  equality (\ref{eq=thm12A}), by applying (\ref{eq=lem8A}), we obtain
\[\Vert\phi_1\mathcal U(NT,0)\chi\phi_2\Vert\leq\frac{C'}{(N+1)\ln^2(N+e)}\]
with $C'>0$ independent of $N$. Let $v$ be the function defined by $(v(t),v_t(t))=\mathcal V(t)g$. Applying (\ref{eq=lem5B}), we can see that $w=(1-\chi)v$ is solution of
\[  \left\{\begin{array}{c}
\partial_t^2w-\Div_x(a\nabla_xw))=[\Delta_x,\chi]v,\\
(w,w_{t})(0,x)=(1-\psi(x))g(x).\end{array}\right.\]
Thus, we get the following representation
\[\mathcal U(NT,0)(1-\chi)=(1-\chi)\mathcal V(NT)- \int_0^{NT}\mathcal U(NT,s) Q\mathcal V(s)\d s,\]
where
\[Q=\left(\begin{array}{cc}0&0\\ 
\cr[\Delta_x,\chi]&0 \end{array}\right).\]
Since $\beta=1$ on $\textrm{supp}\chi$, we can rewrite this representation in the following way
\[\mathcal U(NT,0)(1-\chi)=(1-\chi)\mathcal V(NT)- \int_0^{NT}\mathcal U(NT,s)\beta Q\beta\mathcal V(s)\d s.\]
It follows
\[\begin{array}{lll}\Vert\phi_1\mathcal U(NT,0)(1-\chi)\phi_2\Vert_{\mathcal L({\B})}&\leq&\Vert\phi_1(1-\chi)\mathcal V(NT)\phi_2\Vert_{\mathcal L({\B})}\\ \ &\ &+C\int_0^{NT}\Vert\mathcal \phi_1U(NT,s)\beta \Vert_{\mathcal L({\B})}\Vert\beta\mathcal V(s)\phi_2\Vert_{\mathcal L({\B})}  \d s.\end{array}\]
Estimates  (\ref{eq=lem8A}), (\ref{eq=lem8B}) and (\ref{eq=thm11A}), imply
\begin{equation} \label{eq=thm12B}\begin{array}{lll}\displaystyle\Vert \phi_1\mathcal{U}(NT,0)\phi_2\Vert_{\mathcal{L}({\B})}&\leq&\displaystyle\frac{C}{(N+1)\ln^2(N+e)}\\
\ \\
\displaystyle\ &\ &\displaystyle+C'\int_{0}^{NT}\frac{1}{(NT-s+1)\ln^2(NT-s+e)}\cdot\frac{1}{(s+1)\ln^2(s+e)}\d s\end{array}\end{equation}
and  we get
\[\Vert \phi_1\mathcal U(NT,0)\phi_2\Vert_{\mathcal{L}({\B})}\leq\displaystyle\frac{C}{(N+1)\ln^2(N+1+e)}+\frac{2C_1}{\left(\frac{NT}{2}+1\right)\ln^2\left(\frac{NT}{2}+e\right)},\   N\in\mathbb N.\]
It follows that
\begin{equation} \label{eq=thm12E}\sum_{N=0}^{+\infty}\Vert \phi_1\mathcal U(NT,0)\phi_2\Vert_{\mathcal{L}({\B})}<+\infty.\end{equation}
Applying (\ref{eq=prop2A}) for all $\theta\in\mathbb C$ satisfying $ \textrm{Im} (\theta)>AT$, we obtain
\begin{equation} \label{eq=thm12F}R_{\phi_1,\phi_2}(\theta)=\phi_1(\mathcal{U}(T,0)-e^{-i\theta})^{-1}\phi_2=-e^{i\theta}\sum_{N=0}^\infty \phi_1\mathcal U(NT,0)\phi_2e^{iN\theta}.\end{equation}
The conditions (\ref{eq=thm12E}) and (\ref{eq=thm12F}) imply that the operator $R_{\psi_1,\psi_2}(\theta)$ admits an holomorphic continuation from $\{\theta\in\mathbb{C}\  :\   \textrm{Im}(\theta) \geq A \}$ to
$\{\theta\in\mathbb{C}\  :\   \textrm{Im}(\theta) >  0\}$ and $R_{\psi_1,\psi_2}(\theta)$ admits a continuous extension from $\{\theta\in\mathbb{C}\  :\   \textrm{Im}(\theta)  > 0\}$ to
$\{\theta\in\mathbb{C}\  :\   \textrm{Im}(\theta) \geq  0\}$. The proof is complete.\end{proof}

\end{document}